\documentclass[reqno, 11pt, a4paper]{amsart}
\usepackage{latexsym,ifthen,xspace}
\usepackage{amsmath,amssymb,amsthm}
\usepackage{enumerate, calc, bbm}
\usepackage[usenames,dvipsnames]{xcolor}
\usepackage{paralist}
\usepackage[colorlinks=true, pdfstartview=FitV, linkcolor=blue,
  citecolor=blue, urlcolor=blue, pagebackref=false]{hyperref}
\usepackage[text={33pc,605pt},centering]{geometry}    %11pt
\usepackage{graphicx}
\newtheorem{theorem}{Theorem}[section]
\newtheorem{lemma}[theorem]{Lemma}
\newtheorem{prop}[theorem]{Proposition}
\newtheorem{assumption}[theorem]{Assumption}
\newtheorem{corro}[theorem]{Corollary}

\theoremstyle{definition}
\newtheorem{definition}[theorem]{Definition}

\theoremstyle{remark}
\newtheorem{remark}[theorem]{Remark}

\numberwithin{equation}{section}

\DeclareMathAlphabet{\mathsl}{OT1}{cmss}{m}{sl}
\SetMathAlphabet{\mathsl}{bold}{OT1}{cmss}{bx}{sl}

\newcommand{\al}{\ensuremath{\alpha}}
\newcommand{\be}{\ensuremath{\beta}}
\newcommand{\ga}{\ensuremath{\gamma}}
\newcommand{\de}{\ensuremath{\delta}}

\newcommand{\ze}{\ensuremath{\zeta}}

\renewcommand{\th}{\ensuremath{\theta}}

\newcommand{\ka}{\ensuremath{\kappa}}

\newcommand{\si}{\ensuremath{\sigma}}

\newcommand{\om}{\ensuremath{\omega}}

\newcommand{\ve}{\ensuremath{\varepsilon}}
\newcommand{\vt}{\ensuremath{\vartheta}}
\newcommand{\vr}{\ensuremath{\varrho}}

\newcommand{\vp}{\ensuremath{\varphi}}

\newcommand{\De}{\ensuremath{\Delta}}

\newcommand{\La}{\ensuremath{\Lambda}}
\newcommand{\Si}{\ensuremath{\Sigma}}

\newcommand{\Om}{\ensuremath{\Omega}}
\newcommand{\cB}{\ensuremath{\mathcal B}}
\newcommand{\cC}{\ensuremath{\mathcal C}}

\newcommand{\cE}{\ensuremath{\mathcal E}}
\newcommand{\cF}{\ensuremath{\mathcal F}}
\newcommand{\cG}{\ensuremath{\mathcal G}}

\newcommand{\cL}{\ensuremath{\mathcal L}}

\newcommand{\cN}{\ensuremath{\mathcal N}}
\newcommand{\cO}{\ensuremath{\mathcal O}}
\newcommand{\cP}{\ensuremath{\mathcal P}}
\newcommand{\cQ}{\ensuremath{\mathcal Q}}

\newcommand{\cS}{\ensuremath{\mathcal S}}

\newcommand{\cX}{\ensuremath{\mathcal X}}

\newcommand{\bbN}{\ensuremath{\mathbb N}}

\newcommand{\bbP}{\ensuremath{\mathbb P}}
\newcommand{\bbQ}{\ensuremath{\mathbb Q}}
\newcommand{\bbR}{\ensuremath{\mathbb R}}

\newcommand{\bbZ}{\ensuremath{\mathbb Z}}
\newcommand{\me}{\ensuremath{\mathrm{e}}}

\newcommand{\md}{\ensuremath{\mathrm{d}}}

\newcommand{\scpr}[3]{%
  \ensuremath{%
    \big\langle
      #1, #2
    \big\rangle_{\raisebox{-0ex}{$\scriptstyle\ell^{\raisebox{.1ex}{$\scriptscriptstyle 2$}} (#3)$}}
  }
}
\newcommand{\norm}[3]{%
   \ensuremath{%
     \mathchoice{\big\lVert #1 \big\rVert}
     {\lVert #1 \rVert}
     {\lVert #1 \rVert}
     {\lVert #1 \rVert}_{\raisebox{-.0ex}{$\scriptstyle\ell^{\raisebox{.2ex}{$\scriptscriptstyle #2$}} (#3)$}}
   }
}

\newcommand{\Norm}[2]{%
  \ensuremath{%
    \mathchoice{\big\lVert #1 \big\rVert}
     {\lVert #1 \rVert}
     {\lVert #1 \rVert}
     {\lVert #1 \rVert}_{\raisebox{-.0ex}{$\scriptstyle#2$}}
  }
}

\DeclareMathOperator{\mean}{\mathbb{E}}
\DeclareMathOperator{\Mean}{\mathrm{E}}
\DeclareMathOperator{\prob}{\mathbb{P}} %law random environment
\DeclareMathOperator{\Prob}{\mathrm{P}} %law random walk

\DeclareMathOperator{\supp}{\mathrm{supp}}

\DeclareMathOperator{\osr}{\mathrm{osr}}
\DeclareMathOperator{\osc}{\mathrm{osc}}

\newcommand{\av}[1]{\mathop{\mathrm{av}}(#1)}

\newcommand{\ldef}{\ensuremath{\mathrel{\mathop:}=}}

\newcommand{\indicator}{\mathbbm{1}}

\begin{document}

\title[Quenched local CLT for the dynamic RCM]{Quenched local limit theorem for random walks among time-dependent ergodic degenerate weights}

%    Remove any unused author tags.

%    author one information
\author{Sebastian Andres}
\address{The University of Manchester}
\curraddr{Department of Mathematics,
Oxford Road, Manchester M13 9PL}
\email{sebastian.andres@manchester.ac.uk}
\thanks{}

\author{Alberto Chiarini}
\address{Eindhoven University of Technology}
\curraddr{Department of Mathematics and Computer Science, 5600 MB Eindhoven}
\email{a.chiarini@tue.nl}
\thanks{}

\author{Martin Slowik}
\address{University of Mannheim}
\curraddr{Mathematical Institute, B6, 26, 68159 Mannheim}
\email{slowik@math.uni-mannheim.de %\\ https://orcid.org/0000-0001-5373-5754
}
\thanks{}


\keywords{Random conductance model, invariance principle, percolation, isoperimetric inequality, De Giorgi's iteration. \\ No datasets were generated or analyzed during the current study.}

\date{\today}

\dedicatory{}

\begin{abstract} 
  We establish a quenched local central limit theorem for the dynamic random conductance model on $\bbZ^d$ only assuming ergodicity with respect to space-time shifts and a moment condition.  As a key analytic ingredient we show H\"older continuity estimates for solutions to the heat equation for discrete finite difference operators in divergence form with time-dependent degenerate weights.  The proof is based on De~Giorgi's iteration technique.  In addition, we also derive a quenched local central limit theorem for the static random conductance model on a class of random graphs with degenerate ergodic weights. 
\end{abstract}

\maketitle

%\tableofcontents

\section{Introduction}\label{sec:INTRO}
One of the most studied models for random walks in random environments is the random conductance model (RCM).  Objectives of particular interest are homogenisation results such as invariance principles or stronger local limit theorems for the associated heat kernel.  For instance, in~\cite{ADS16} a local  limit theorem has been proven for random walks under general ergodic conductances satisfying a certain moment condition.

For the dynamic RCM evolving in a time-varying random environment a local limit theorem has been stated in \cite{An14} which required uniform ellipticity, meaning that the conductances are almost surely uniformly bounded and bounded away from zero, as well as polynomial mixing, i.e.\ the polynomial decay of the correlations of the conductances in space and time.  In this paper we significantly relax these assumptions and show a quenched local limit theorem for the dynamic RCM with degenerate space-time ergodic conductances that only need to satisfy a moment condition.  In contrast to many results on various models for random walks in dynamic random environments, in the present paper the environment is \emph{not} assumed to be uniformly elliptic or mixing or Markovian in time and we also do not require any regularity with respect to the time parameter. 

The proof exploits a quenched invariance principle established under the same assumptions in \cite{ACDS18}.  In addition and original to this paper, some H\"older continuity in the macroscopic scale for the heat kernel is required.  For the proof we extend the De~Giorgi iteration technique to discrete finite-difference divergence-form operators with time-dependent degenerate coefficients.  De~Giorgi iteration is an alternative to the well-known Moser iteration.  The latter has been implemented for the discrete graph setting in \cite{De99, ADS16}.  It turns out that the De~Giorgi's iteration method performs far more efficiently for proving H\"older regularity of time-space harmonic functions.  On one hand, it avoids the need for a parabolic Harnack inequality in contrast to the arguments in \cite{De99, ADS16}, and it also makes the proof significantly simpler and shorter. 

\subsection{Setting and main result}
\label{sec:setting_intro}
Consider the  Euclidean lattice, $(\bbZ^d, E_d)$, for $d \geq 2$, whose edge set, $E_d$, is given by the set of all non-oriented nearest neighbour bonds, that is $E_d = \{ \{x,y\} :  x,y \in \bbZ^d,\ |x-y| = 1 \}$.    The graph $(\bbZ^d, E_d)$ is endowed  with a family of time-dependent positive weights $\om \equiv \{\om_t(e) : e \in E_d,\, t \in \bbR \}$.  We refer to $\om_t(e)$ as the \emph{conductance} of an edge $e$ at time $t$.  Let $\Om$ be the set of measurable functions from $\bbR$ to $(0,\infty)^{E_d}$ equipped with a $\si$-algebra $\cF$ and let $\prob$ be a probability measure on $(\Om, \cF)$.  We  write $\mean$ for the expectation with respect to $\prob$.  Upon $\Om$ we consider the $d+1$-parameter group of translations $(\tau_{t,x} : (t,x)\in \bbR \times \bbZ^d)$ given by
\begin{align} \label{eq:shift:space-time}
  \tau_{t,x}\!:\Om \rightarrow \Om,
  \qquad
  \big\{\om_s(e) : (s,e) \in \bbR \times E_d\big\}
  \;\longmapsto\;
  \big\{\om_{t+s}(e+x) : (s,e) \in \bbR \times E_d\big\}.
\end{align}
\begin{assumption}\label{ass:P}
  \begin{enumerate}[(i)]
  \item $\prob$ is ergodic and stationary with respect to space-time shifts, that is, for all $x \in \bbZ^d$, $t\in \bbR$,  $\prob \circ\, \tau_{t,x}^{-1} \!= \prob\,$ , and $\prob[A] \in \{0,1\}\,$ for any $A \in \cF$ such that $\prob[A \triangle \tau_{t,x}(A)] = 0$ for all $x \in \bbZ^d$, $t\in \bbR$.
  
  \item For every $A \in \cF$ the mapping $(\om,t,x)\mapsto \indicator_A(\tau_{t,x}\om)$ is jointly measurable with respect to the $\si$-algebra $\cF \otimes \cB(\bbR)\otimes \cP(\bbZ^d)$.

  \item $\mean\big[\om_t(e)\big]< \infty$ and $\mean\big[\om_t(e)^{-1}\big] < \infty$ for any $e \in E_d$ and $t \in \bbR$.
  \end{enumerate}
\end{assumption}
For a given $\om \in \Om$ and for $s \in \bbR$ and $x \in \bbZ^d$, let $\Prob_{s,x}^{\om}$ be the probability measure on the space of $\bbZ^d$-valued c\`{a}dl\`{a}g functions on $\bbR$, under which the coordinate process $X \equiv (X_t : t \in \bbR)$ is the time-inhomogeneous Markov process on $\bbZ^d$ starting in $x$ at time $s$ with time-dependent generator (in the $L^2$-sense) acting on bounded functions $f\!: \bbZ^d \to \bbR$ as
\begin{align*}
  \big(\cL_t^{\om} f\big)(x)
  \;=\;
  \sum_{y: |x-y|=1}\mspace{-6mu} \om_t(\{x, y\}) \, \big(f(y) \,-\, f(x)\big).
\end{align*}
In other words, $X$ is the continuous-time random walk with time-dependent jump rates given by the conductances, i.e.\ the random walk $X$ chooses its next position at random proportionally to the conductances.  Note that the total jump rate out of any lattice site is not normalised, and the law of the sojourn time of $X$ depends on its time-space position.  Therefore, $X$ is often called the \emph{variable speed random walk (VSRW)}.  It is known that under Assumption~\ref{ass:P}-(iii) the process $X$ does not explode, i.e.\ there are only finitely many jumps in finite time, see \cite[Lemma~4.1]{ACDS18}.  Note that the counting measure is a time-independent invariant measure for $X$.  For $x,y \in \bbZ^d$ and $t\geq s$, we denote $p^{\om}(s,x;t,y)$ the heat kernel of $(X_t : t \geq s )$, that is
\begin{align*}
  p^{\om}(s,x;t,y)
  \;\ldef\;
  \Prob_{s,x}^{\om}\big[X_t = y\big].
\end{align*}
During the last decade, considerable effort has been invested in the derivation of a quenched functional central limit theorem (QFCLT) or quenched invariance principle, see the surveys \cite{Bi11, Ku14} (and references therein), and \cite{ADS15, BS19, DNS18} for more recent results on the static RCM.  For RCMs including long-range jumps a QFCLT has been recently established in  \cite{BCKW20} .  For the time-dynamic RCM with ergodic degenerate conductances the following QFCLT has been shown in \cite{ACDS18}.  We refer to \cite{BR18} for a closely related result including random walks on dynamical bond percolation.
\begin{assumption}\label{ass:moment}
  There exist $p, q \in (1, \infty]$ satisfying
  \begin{align*}
    \frac{1}{p-1} \,+\, \frac{1}{(p-1) q} \,+\, \frac{1}{q}
    \;<\;
    \frac{2}{d}
  \end{align*}
  such that for any $e \in E_d$ and $t \in \bbR$,
  \begin{align*}
    \mean\!\big[\om_t(e)^p\big] \;<\; \infty
    \quad \text{and} \quad
    \mean\!\big[\om_t(e)^{-q}\big] \;<\; \infty.
  \end{align*}  
\end{assumption}
\begin{theorem}[QFCLT \cite{ACDS18}]\label{thm:dyn_ip}
  Suppose that Assumptions~\ref{ass:P} and \ref{ass:moment} hold.  Then, for $\prob$-a.e.\ $\om$, the process $X^{(n)} \equiv \big(X_t^{(n)} \ldef n^{-1} X_{n^2 t}: t \geq 0 \big)$ converges (under $\Prob_{\!0}^\om$) in law towards a Brownian motion on $\bbR^d$ with a deterministic non-degenerate covariance matrix $\Si^2$.
\end{theorem}
As our main result we establish a quenched local limit theorem (or quenched local CLT) for $X$, which states that, $\prob$-a.s., under diffusive scaling the rescaled transition densities converge uniformly over compact sets towards the Gaussian transition density of the Brownian motion with covariance matrix $\Sigma^2$ appearing as the limit process in Theorem~\ref{thm:dyn_ip}. That Gaussian density will be denoted
\begin{align} \label{eq:def_kt}
  k_t(x)
  \;\equiv\;
  k_t^{\Si}(x)
  \;\ldef\;
  \frac{1}{\sqrt{(2\pi t)^d \det \Si^2}}
  \exp\!\left(-x \cdot(\Si^2)^{-1}x/2t\right).
\end{align}
\begin{theorem}[Quenched local CLT] \label{thm:dyn_lclt} 
  Suppose that Assumptions~\ref{ass:P} and \ref{ass:moment} hold.  For any $T_2> T_1 > 0$ and $K > 0$,
  \begin{align*}
    \lim_{n \to \infty} \sup_{|x|\leq K} \sup_{ t\in [T_1, T_2]}
    \big| n^d \, p^{\om}(0, 0; n^2 t, \lfloor nx \rfloor) - k_t(x) \big|
    \;=\;
    0,
    \qquad \text{for }\prob\text{-a.e. }\om.
  \end{align*}
\end{theorem}
In general, a local limit theorem is a stronger statement than a FCLT.  In fact, even in the case of time independent i.i.d.\ conductances, where the QFCLT is known to hold \cite{ABDH13}, the heat kernel may behave subdiffusively due to a trapping phenomenon (see \cite{BBHK08}), so that a local limit theorem may fail in general.  Nevertheless it does hold, for instance, in the case of uniformly elliptic conductances or for random walks on supercritical i.i.d.\ percolation clusters, see \cite{BH09}.  We refer to \cite{BKM15} for sharp conditions on the tails of i.i.d.\ conductances at zero for Harnack inequalities and a local limit theorem to hold. 
Stronger quantitative homogenization results for heat kernels and Green functions can be established by using techniques from quantitative stochastic homogenization, see \cite[Chapters 8--9]{AKM19} for details in the uniformly elliptic case. This technique has been adapted to the VSRW on static percolation clusters in \cite{DG19}, and it is expected that it also applies to other degenerate models.
 In the general ergodic setting it is known that moment conditions are necessary even for the QFCLT to hold (cf.\ \cite{BBT16}).  In fact, in \cite{ADS16,AT19} quenched local limit theorems have been derived under moment conditions that turned out to be optimal in certain cases.  A corresponding result for a class of symmetric diffusions has been obtained in \cite{CD15}.

Since the static RCM is naturally included in the time-dynamic model, the moment condition in Assumption~\ref{ass:moment} is not optimal for both, the QFCLT and local limit theorem. For the static VSRW, a QFCLT holds in $d=2$ already under the moment condition with $p=q=1$ (see \cite{Bi11}), a local limit theorem has recently been shown in \cite{BS20} under  the moment condition with $1/p+1/q<2/(d-1)$, which is a weaker condition on $p$ and $q$ as the one in Assumption~\ref{ass:moment}.

Relevant examples for dynamic RCMs include random walks in an environment generated by some interacting particle systems like zero-range or exclusion processes, cf.\ \cite{MO16}.  Some on-diagonal heat kernel upper bounds for a degenerate time-dependent conductances model are obtained in \cite{MO16}.  Full two-sided Gaussian estimates are known in the uniformly elliptic case for the VSRW \cite{DD05} or for constant speed walks under effectively non-decreasing conductances \cite{DHZ19}.  However, unlike for static environments, two-sided Gaussian heat kernel bounds are much less regular and some pathologies may arise as they are not stable under perturbations, see \cite{HK16}. Moreover, such bounds are expected to be governed by a time-dynamic version of the intrinsic distance whose exact form in a degenerate setting is unknown (cf.\ e.g.\ \cite{ADS19} for some results on the static RCM).  These facts make the derivation of Gaussian bounds for the dynamic RCM with unbounded conductances a subtle open challenge.

Finally, let us remark that there is a link between the time dynamic RCM and Ginzburg-Landau $\nabla\vp$ interface models as such random walks appear in the so-called Helffer-Sj\"ostrand representation of the space-time covariance in these models (cf.\ \cite{DD05, AT19}).  In this context, the annealed heat kernel of such a dynamic RCM is relevant.  Although the quenched version in Theorem~\ref{thm:dyn_lclt} does not directly imply an annealed local limit theorem, such a result has recently been shown in \cite{AT19} under a stronger moment condition (the proof relies on the quenched version in Theorem~\ref{thm:dyn_lclt}), which is then applied in \cite[Section~5]{AT19} to obtain a scaling limit for the space-time covariances  in the Ginzburg-Landau $\nabla\varphi$ model.  This result also applies to interface models with certain convex but not strictly convex potentials.

\subsection{The method}
The proof of Theorem~\ref{thm:dyn_lclt} has two non-trivial main ingredients, the invariance principle in Theorem~\ref{thm:dyn_ip} and a H\"older regularity estimate for the heat kernel.  For the latter it is common to use a purely analytic approach and to interpret the heat kernel as a fundamental solution of the heat equation
\begin{align}\label{eq:heat}
  (\partial_t - \cL^\om_t) u \;=\; 0.
\end{align}
Then the aim becomes a regularity estimate at large scales for solutions to the parabolic equation \eqref{eq:heat} with weights $\om$ which are not uniformly bounded away from zero and infinity.  As observed in \eqref{eq:def:generator} below, $\cL^\om_t f(x) = -\nabla^*(\om_t\nabla f)$ is in divergence form and thus it may be regarded as the discrete analogue to the operator $(L^a_t f)(x) = \sum_{i,j = 1}^d \partial_{x_i} \big(a_{ij}(t,x) \partial_{x_j} f(x)\big)$, acting on functions on $\bbR^d$, where $a = (a_{ij}(t,x))$ is a time-dependent symmetric positive definite matrix.  The question about regularity of solutions to the continuous heat equation $(\partial_t - L^a_t) u = 0$ is very classical.  The first results appeared independently in the influential works by De~Giorgi~\cite{dG57} and Nash~\cite{Na57}.  They showed that solutions to elliptic or parabolic problems are H\"older continuous if the coefficient matrix $a$ is uniformly elliptic.  Later, a new and farther reaching proof was provided by Moser~\cite{Mo60}.  In fact, nowadays the by far most common approach is to deduce H\"older regularity from a parabolic Harnack inequality (PHI) derived by Moser's iteration technique. In the continuous setting this has been implemented in \cite{KK77} for parabolic equations with time-dependent degenerate coefficients.  In the case of static and normalised weights on graphs, the approach has been used in \cite{De99} for uniformly elliptic weights and in \cite{ADS16} for degenerate weights satisfying an integrability condition.  However, in the present setting the approach fails. Indeed, the most difficult step in the proof of the PHI is to link a certain $\ell^{\al}$-norm of $u$ with its $\ell^{-\al}$-norm (cf.\ \cite[Section~4.2]{ADS16} or \cite[Section~2.4]{De99}). Those arguments require maximal inequalities on a whole range of space-time cylinders.  Unless the weights are normalised, due to certain effects on discrete spaces such maximal inequalities can only be derived between time-space cylinders on certain scales (manifested in the lower bound on $\si - \si'$ in the maximal inequality in Theorem~\ref{thm:max:ineq} below), which is not sufficient to derive a PHI.

To circumvent those obstructions we take a different route and revisit the original method of De~Giorgi \cite{dG57} and transfer it to the discrete equation~\eqref{eq:heat} on a certain class of graphs while we allow the weights $\om$ to be unbounded.  However, in a central step in \cite[Lemma~II]{dG57}, see also \cite[Equation~(5.5)]{LSU68}, the level sets of a solution are controlled by an application of an isoperimetric inequality, which fails in our setting of a discrete gradient associated with the non-local operator $\cL^\om_t$.  Instead, following an idea in \cite{WYW06}, we control the level sets of a solution $u$ to~\eqref{eq:heat} by bounding their sizes in terms of $(-\ln u)_+$ (see Lemmas~\ref{lemma:levelset:apriori} and \ref{lemma:levelset:small} below).  Then, the key result is an oscillation inequality stated in Theorem~\ref{thm:osc} below, which directly implies H\"older regularity.  Since we do not assume any uniform upper or lower bound on the conductances $\om_t(x,y)$, the global upper and lower bounds on $\om_t(x,y)$ need to be replaced by certain integrability conditions on $\om_t(x,y)$ and $1/\om_t(x,y)$.  Although this procedure does not require a full PHI, it still provides a weak PHI, see Theorem~\ref{thm:weak:harnack} below.

\subsection{Random walks on random graphs}
As an additional result we derive in Section~\ref{sec:rg} a local limit theorem for random walks evolving on a random graph under static ergodic random conductances satisfying a similar moment condition, see Theorem~\ref{thm:lclt_rg} below. Our assumptions cover a certain class of random graphs including supercritical i.i.d.\ percolation clusters and clusters in percolation models with long range correlations, see e.g.\ \cite{DRS14, Sa17}.  The corresponding QFCLT has been shown in \cite{DNS18}.  In fact, the oscillation inequality in Theorem~\ref{thm:osc} is sufficiently robust so that Theorem~\ref{thm:lclt_rg} can be derived from it by similar arguments as Theorem~\ref{thm:dyn_lclt}.

\subsection{Structure of the paper}
In Section~\ref{sec:Sobolev} we implement the De~Giorgi iteration and show the oscillation inequality.  In Section~\ref{sec:localclt} we establish in Theorem~\ref{thm:lclt_det} a local limit theorem for random walks on a class of subgraphs of $\bbZ^d$, provided a H\"older continuity estimate at large scales holds.  Then this is used to show Theorem~\ref{thm:dyn_lclt} in Section~\ref{sec:dyn_lattice}.  The result for random walks on random graphs is discussed in Section~\ref{sec:rg}.   Appendix~\ref{sec:tech} contains a technical lemma needed in the proofs, and in Appendix~\ref{sec:be_fe} we verify the forward and backward equations for the transition semigroup of $X$.

\section{De Giorgi iteration on graphs}
\label{sec:Sobolev}

\subsection{Setting and notation}
In this section we will work in a more general deterministic framework. We consider an infinite, connected, locally finite graph $G = (V, E)$   with vertex set $V$ and non-oriented edge set $E$.  We write $x \sim y$ if $\{x,y\} \in E$.  We endow the graph $(V,E)$ with time-dependent, positive weights $\om = \{\om_t(e) \in (0, \infty) : e \in E, t \in \bbR\}$, where for each $e\in E$ the map $t\mapsto \om_t(e)$ is assumed to be measurable.  Next we introduce the time-dependent finite-difference operator
\begin{align}\label{eq:operator}
  \cL_t^\om f(x)
  \;=\;
  \sum_{y \sim x} \om_t(\{x, y\}) \, \big( f(y) \,-\, f(x)\big),
  \qquad t \in \bbR,\, x \in V,
\end{align}
acting on bounded functions $f\!:V \to \bbR$.  Further, we define the measures $\mu_t^{\om}$ and $\nu_t^{\om}$ on $V$ by
\begin{align*}
  \mu_t^{\om}(x) \;\ldef\; \sum_{y \sim x}\, \om_t(\{x,y\})
  \qquad \text{and} \qquad
  \nu_t^{\om}(x) \;\ldef\; \sum_{y \sim x}\, \frac{1}{\om_t(\{x,y\})}.
\end{align*}
We endow $(V,E)$ with the counting measure that assigns to any $A \subset V$ the number $|A|$ of elements in $A$.  Moreover, we denote by $B(x,r) \ldef \{y \in V: d(x,y) \leq \lfloor r \rfloor\}$ the closed ball with center $x$ and radius $r$ with respect to the natural graph distance $d$, and for a set $A \subset V$ we define its boundary by $\partial A \ldef \{ x\in A \,:\, \exists\, y \in V \setminus A \text{ such that } \{x,y\} \in E \}$.  For functions $f\!:A \to \bbR$, where either $A \subset V$ or $A \subset E$, the $\ell^p$-norm $\norm{f}{p}{A}$ will be taken with respect to the counting measure.  The corresponding scalar products in $\ell^2(V)$ and $\ell^2(E)$ are denoted by $\langle \cdot, \cdot \rangle_{\ell^2(V)}$ and $\langle \cdot, \cdot \rangle_{\ell^2(E)}$, respectively.  For any non-empty, finite $B\subset V$ and $p\in (0,\infty)$, we introduce space-averaged norms on functions $f\!: B \to \bbR$ by 
\begin{align*}
  \Norm{f}{p,B}
  \;\ldef\;
  \bigg(\frac{1}{|B|}\sum_{x\in B} |f(x)|^p\bigg)^{\!\!1/p}.
\end{align*}
Moreover, for any non-empty compact interval $I \subset \bbR$ and any finite $B \subset V$ and $p, p' \in (0, \infty)$, we define space-time-averaged norms on functions $u\!: I \times B \to \bbR$ by
\begin{align*}
  \Norm{u}{p, p', I \times B}
  \;\ldef\;
  \bigg(
    \frac{1}{|I|}\; \int_I\, \Norm{u_t}{p, B}^{p'}\; \md t
  \bigg)^{\!\!1/p'}\!\!
  \ldef\;
  \bigg(
    \frac{1}{|I|}\;
    \int_I\,
      \bigg(
        \frac{1}{|B|}\, \sum_{x \in B}\, |u_t(x)|^p
      \bigg)^{\!p'/p}
    \md t
  \bigg)^{\!\!1/p'}
\end{align*}
and $\Norm{u}{p, \infty, I \times B} \ldef \sup_{t \in I} \Norm{u_t}{p, B}$,
% %
% \begin{align*}
%   \Norm{u}{p, \infty, I \times B}
%   \;\ldef\;
%   \sup_{t \in I} \Norm{u_t}{p, B},
% \end{align*}
% %
where $u_t( \cdot ) \ldef u(t,  \cdot )$ for any $t \in I$.
\medskip

Next we need to introduce some discrete calculus. For $f\!: V \to \bbR$ and $F\!: E \to \bbR$ we define the operators $\nabla f\!: E \to \bbR$ and $\nabla^*F\!: V \to \bbR$ by
\begin{align*}
  \nabla f(e) \;\ldef\; f(e^+) - f(e^-),
  \qquad \text{and} \qquad
  \nabla^*F (x)
  \;\ldef\;
  \sum_{e: e^+ =\,x}\! F(e) \,-\! \sum_{e:e^-=\, x}\! F(e),
\end{align*}
where for each non-oriented edge $e \in E$ we specify one of its two endpoints as its initial vertex $e^+$ and the other one as its terminal vertex $e^-$.  Nothing of what will follow depends on the particular choice.  Since $\scpr{\nabla f}{F}{E} = \scpr{f}{\nabla^* F}{V}$ for all $f \in \ell^2(V)$ and $F \in \ell^2(E)$, $\nabla^*$ can be seen as the adjoint of $\nabla$.  Notice that in the discrete setting the product rule reads
\begin{align}\label{eq:rule:prod}
  \nabla(f g)
  \;=\;
  \av{f} \nabla g \,+\, \av{g} \nabla f,
\end{align}
where $\av{f}(e) \ldef \frac{1}{2}(f(e^+) + f(e^-))$.  We observe that the operator $\cL_t^{\om}$ defined in~\eqref{eq:operator} has the form
\begin{align}\label{eq:def:generator}
  \cL^{\om}_t f (x)
  \; = \;
  - \nabla^*(\om_t \nabla f) (x).
\end{align}
For any $t \in \bbR$, the \emph{time-dependent Dirichlet form} associated with $\cL_t^{\om}$ is given by
\begin{align} \label{eq:def:dform}
  \cE_t^{\om}(f,g)
  \;\ldef\;
  \scpr{f}{-\cL_t^{\om} g}{V}
  \;=\;
  \scpr{\nabla f}{\om_t \nabla g}{E},
%  \;=\;
%  \scpr{1}{\md \Ga_t^{\om}(f,g)}{E},
\end{align}
and we set $\cE_t^{\om}(f) \ldef \cE_t^{\om}(f,f)$.

Finally, throughout the paper, we write $c$ to denote a positive, finite constant which may change on each appearance.  Constants denoted by $C_i$ will remain the same.  In this section we make the following assumptions on the graph $(V,E)$.
\begin{assumption}\label{ass:graph}
  Let $d \geq 2$.  There exist constants $c_{\mathrm{reg}}, C_{\mathrm{reg}}, C_{\mathrm{S}_1}, C_{\mathrm{P}} \in (0, \infty)$ and $ C_{\mathrm{W}} \in [1,\infty)$ such that for all $x \in V$ the following hold.
  \begin{itemize}
  \item[(i)] Volume regularity of order $d$ for large balls. There exists $N_1(x) < \infty$ such that for all $n \geq N_1(x)$,
    \begin{align}\label{eq:reg:large_balls}
      c_{\mathrm{reg}}\, n^d \;\leq\; |B(x, n)| \;\leq\; C_{\mathrm{reg}}\, n^d.
    \end{align}

  \item[(ii)] Sobolev inequality.   There exist $d'\geq d$ and  $N_2(x) < \infty$ such that for all $n \geq N_2(x)$,
    \begin{align} \label{eq:sobolev}
      \Norm{u}{d'/(d'-1),B(x,n)}
      \;\leq\;
      C_{\mathrm{S_1}}\, \frac{ n}{|B(x, n)|}\, \Norm{\nabla u}{\ell^1(E)},
    \end{align}
    for every function $u\!: V \to \bbR$ with $\supp u \subset B(x, n)$.  
    
  \item[(iii)] Weak Poincar\'{e} inequality.  There exists $N_3(x) < \infty$
    such that for all \mbox{$n \geq N_3(x)$,}
    \begin{align}\label{eq:poincare:ineq}
      \sum_{y \in B(x,n)}\mspace{-6mu} \big| u(x) - (u)_{\cN} \big|
      \;\leq\;
      C_{\mathrm{P}}\, n\, \bigg( 1 + \frac{|B(x,n)|}{|\cN|} \bigg)\,
      % \frac{n}{|B(x, C_{\mathrm{W}} n)|}\,
      \sum_{\substack{y, y' \in B(x, C_{\mathrm{W}} n)\\ y \sim y'}}
      \mspace{-24mu}|\nabla u_t(\{y, y'\}|
    \end{align}
    for every $u\!: V \to \bbR$ and $\cN \subset B(x, n)$ where $(u)_{\cN} \ldef \frac{1}{|\cN|} \sum_{y \in \cN} u(y)$.
  \end{itemize}
\end{assumption}
\begin{remark} \label{rem:ass_graph}
  \begin{enumerate}[(i)]
  \item The Euclidean lattice, $(\bbZ^d, E_d)$, satisfies Assumption~\ref{ass:graph} with $d'=d$ and $N_1(x)=N_2(x)=N_3(x)=1$. 

  \item Suppose that Assumption~\ref{ass:graph}-(i) holds. Then the Sobolev inequality in \ref{eq:sobolev} follows from an isoperimetric inequality for large sets, see \cite[Proposition~3.5]{DNS18}. The weak Poincar\'{e} inequality in \eqref{eq:poincare:ineq} follows from a classical local $\ell^1$-Poincar\'{e} inequality, which in turn can be obtained from a (weak) relative isoperimetric inequality by applying a discrete version of the co-area formula, see \cite[Lemma~3.3.3]{Sa96}.
  \end{enumerate}
\end{remark}
Next we recall that Assumption~\ref{ass:graph} implies a weighted space-time Sobolev inequality for functions with compact support. Noting that $d' \geq d\geq 2$ we set
\begin{align}\label{eq:def:rho}
  \rho
  \;\equiv\;
  \rho(d',  q)
  \;=\;
  \frac{d'}{d' - 2 + d'/q}.
%  \;=\;
%  \frac{d}{d-2(1-\th)/(1-\th/d) + d/q}.
\end{align}
\begin{prop}[Space-time Sobolev inequality]\label{prop:sobolev}
  Suppose that Assumption~\ref{ass:graph}-(i) and (ii) hold for some $d' \geq d$. %$\th \in [0, 1)$.
  Let $I \subset \bbR$ be a compact interval.  Then, for any $q \in [1, \infty)$, $q' \in [1, \infty]$ there exists $C_{\mathrm{S}} \equiv C_{\mathrm{S}}(d, \th, q) < \infty$ such that for any $x \in V$ and $n \geq N_1(x) \vee N_2(x)$,
  \begin{align}\label{eq:sob:ineq:weight}
    \Norm{u^2}{\rho, q'/(q'+ 1), I \times B(x, n)}
    \;\leq\;
    C_{\mathrm{\,S}}\, n^2\, \Norm{\nu^{\om}}{q, q'\!, I \times B(x, n)}\;
    \bigg(
      \frac{1}{|I|}\,
      \int_I\, \frac{\cE_{t}^{\om}(u_t)}{|B(x, n)|}\; \md t
    \bigg)
  \end{align}
  for every $u\!: \bbR \times V \to \bbR$ with $\supp u \subset I \times B(x, n)$.  If $\th > 0$, then \eqref{eq:sob:ineq:weight} holds for $q = \infty$.
\end{prop}
\begin{proof}
  See \cite[Proposition~5.4]{ACDS18}.
\end{proof}

\subsection{H\"older regularity estimates}
Our main objective in this section is to implement De Giorgi's iteration scheme in the graph setting with time-dependent degenerate weights in order to derive a H\"older regularity estimate for solutions of parabolic equations.  Write 
\begin{align*}
  \osc_Q  u
  \;\ldef\;
  \sup_{(t, x) \in Q} u(t,x) - \inf_{(t,x) \in Q} u(t,x)
\end{align*}
for the oscillation a the function $u$ over a set $Q \subset \bbR \times V$.
\begin{theorem}[Oscillation inequality]\label{thm:osc}
  Suppose that Assumption~\ref{ass:graph} holds.  For $t_0 \in \bbR$, $x_0 \in V$ and $n \geq N_4(x_0)\ldef 2^{8d} C_{\mathrm{W}} \max\{ N_1(x_0), N_2(x_0), N_3(x_0)\}$, let $u > 0$ be such that $\partial_t u - \cL_t^{\om} u = 0$ on $Q(n) =[t_0-n^2,t_0]\times B(x_0,n)$.  Then, for any $p, p', q, q' \in [1, \infty]$ satisfying
  \begin{align}\label{eq:cond:pq2}
    \frac{1}{p} \cdot \frac{p'}{p'-1} \cdot \frac{q'+1}{q'}
    \,+\,
    \frac{1}{q}
    \;<\;
    \frac{2}{d'},
  \end{align}
  and such that $\Norm{1 \vee \mu^{\om}}{p,p',Q(n)}, \Norm{1 \vee \nu^{\om}}{q,q',Q(n)} < \infty$, there exist  $\vartheta \in (0, 1/2)$ and
  \begin{align*}
    \ga^{\om}
    \;=\;
    \ga^{\om}(x_0,n)
    \;=\;
    \ga\big(\Norm{1 \vee \mu^{\om}}{p,p',Q(n)},
    \Norm{1 \vee \nu^{\om}}{q,q',Q(n)}\big)
    \in
    (0, 1), 
  \end{align*}
  where $\ga\!: [0, \infty)^2 \to (0,1)$ is continuous and increasing in both components, such that
  \begin{align}\label{eq:oscillation_ineq}
    \osc_{Q(\vt n)} u
    \;\leq\;
    \ga^{\om}\, \osc_{Q(n)} u.
  \end{align}
\end{theorem}
Theorem~\ref{thm:osc} will be proven in Section~\ref{sec:proof_osc} below. This oscillation inequality becomes particularly interesting when $\om$ is a random weight configuration on $\bbZ^d$ satisfying Assumptions~\ref{ass:P} and \ref{ass:moment}. Then, by the ergodic theorem,  \eqref{eq:oscillation_ineq} holds with the same constant $\bar{\ga}$ for all $n \in \bbN$ large enough. 
\begin{assumption} \label{ass:lim_const}
  For any $\de > 0$, $\sqrt{t_0}/2 > \de$ and $x_0 \in V$ there exists $N_5(x_0, t_0) < \infty$ such that  
  \begin{align*}
    \bar{\mu}
    &\;\ldef\;
    \sup_{n \geq N_5(x_0, t_0)}
    \Norm{\mu^{\om}}{p, p', n^2 [t_0 - \de^2, t_0] \times B(x_0, \de n)}
    \;<\;
    \infty,
    \\[.5ex]
    \bar{\nu}
    &\;\ldef\;
    \sup_{n \geq N_5(x_0, t_0)}
    \Norm{\nu^{\om}}{p, p', n^2 [t_0 - \de^2, t_0] \times B(x_0, \de n)}
    \;<\;
    \infty
  \end{align*}
  are independent of $\de$, $x_0$ and $t_0$.  Write $\bar{\ga} = \ga(\bar \mu, \bar \nu) \in (0,1)$.
\end{assumption}
Assumption~\ref{ass:lim_const} is satisfied, for instance, on the lattice $\bbZ^d$ under Assumptions~\ref{ass:P} and \ref{ass:moment}, cf.\ Proposition~\ref{prop:krengel_pyke} below.
\begin{corro} \label{cor:cont_caloric}
  Suppose that Assumptions~\ref{ass:graph} and \ref{ass:lim_const} hold.  For any $\de > 0$, $x_0 \in V$ and $\sqrt{t_0}/2 > \de$ fixed, let  $\bar{\ga}$ be as in Assumption \ref{ass:lim_const}.  Further, let $n \in \bbN$ be such that $\de n \geq N_4(x_0) \vee N_5(x_0,t_0)$.  Suppose that $u > 0$ is such that $\partial_t u - \cL_t^{\om} u = 0$  on $[0, t_0] \times B(x_0, n)$.  Then, for any $t\in n^2 [t_0 - \de^2, t_0]$ and $x_1, x_2 \in B(x_0, \de n)$,
  \begin{align*}
    \big|u(t,x_1) - u(t,x_2)\big|
    \;\leq\;
    C_1\, \bigg( \frac{\de}{\sqrt{t_0}} \bigg)^{\!\!\vr}\,
    \max_{[3 t_0 / 4, t_0] \times B(x_0, \sqrt{t_0} / 2)} u,
  \end{align*}
  where $\vr \ldef \ln \bar{\ga} / \ln \vt$ and $C_1$ only depends on $\bar{\ga}$.
\end{corro}
\begin{proof}
  Set $\de_k \ldef \vt^k \sqrt{t_0} /2$, $k \geq 0$ and, with a slight abuse of notation, let
  \begin{align*}
    Q_k
    \;\ldef\;
    n^2 \big[t_0- \de_k^2, t_0\big] \times B\big(x_0, \de_k n\big),
    \qquad k \geq 0.
  \end{align*}
  Choose $k_0 \in \bbN$ such that $\de_{k_0} \geq \de > \de_{k_0+1}$.  In particular, for every $k \leq k_0$ we have $\de_{k} \in [\de, \sqrt{t_0}]$.  Now we apply Theorem~\ref{thm:osc} and Assumption~\ref{ass:lim_const}, which gives
  \begin{align*}
    \osc_{Q_{k}} u
    \;\leq\;
    \bar{\ga}\, \osc_{Q_{k-1}} u ,
    \qquad \forall\, k = 1, \ldots, k_0.
  \end{align*}
  We iterate the above inequality on the chain $Q_0 \supset Q_1 \supset \cdots \supset Q_{k_0}$ to obtain
  \begin{align} \label{eq:pre_hoelder}
    \osc_{Q_{k_0}} u
    \;\leq\;
    \bar{\ga}^{k_0}\, \max_{Q_{0}} u.
  \end{align}
  Note that
  \begin{align*}
    Q_{k_0}
    \;=\;
    n^2 \big[t_0 - \de_{k_0}^2, t_0\big] \times B\big(x_0, \de_{k_0} n\big)
    \;\supset\;
    n^2 \big[t_0 - \de^2, t_0\big] \times B\big(x_0, \de n\big).
  \end{align*}
  Hence, since $\bar{\ga}^{k_0} \leq c \big(\de / \sqrt{t_0}\big)^\vr$, the claim follows from \eqref{eq:pre_hoelder}.
\end{proof}

\subsection{Maximal inequality}
For any $x_0 \in V$, $t_0 \geq 0$ and $n \geq 1$, denote by $Q(n) \equiv Q(t_0, x_0, n) \ldef [t_0-n^2, t_0] \times B(x_0, n)$ the corresponding time-space cylinder, and set
\begin{align*}
  Q_{\tau, \si}(n)
  \;\equiv\;
  Q_{\tau, \si}(t_0, x_0, n)
  \;\ldef\;
  [t_0 -\tau n^2, t_0] \times B(x_0, \si n),
  \qquad \si, \tau \in [0,1],
\end{align*}
and  $Q_{\si}(n) \equiv Q_{\si,\si}(n)$. In this subsection we will show the following maximal inequality as the main result.
\begin{theorem}\label{thm:max:ineq}
  Let Assumption~\ref{ass:graph}-(i) and (ii) be satisfied.  For $t_0 \in \bbR$, $x_0 \in V$ and $n \geq 2(N_1(x_0) \vee N_2(x_0))$, suppose that $u$ is such that $\partial_t u - \cL_t^{\om} u = 0$ on $Q(n)$.  Then, for any $0 \leq \De < 2/(d+2)$ and $p, p', q, q' \in [1, \infty]$ satisfying
  \begin{align}\label{eq:cond:pq}
    \frac{1}{p} \cdot \frac{p'}{p'-1} \cdot \frac{q'+1}{q'}
    \,+\,
    \frac{1}{q}
    \;<\;
    \frac{2}{d'},
  \end{align}
  there exist $\ka \equiv \ka(p, p', q, q', d') \in (0, \infty)$ and $C_2 \equiv C_2(p, p', q, q', d') < \infty$ such that for all $h \geq 0$ and $1/2 \leq \si' < \si \leq 1$ with $\si - \si' > 4 n^{-\De}$,
  \begin{align}\label{eq:max:ineq}
    &\max_{(t,x) \in Q_{\si'}(n)}\! u(t,x)
    %\nonumber
    \\[.5ex]
    &\mspace{36mu}\leq\;
    h \,+\,
    C_2\,
    \bigg(
      \frac{\Norm{1 \vee \mu^{\om}}{p, p'\!, Q(n)}\,
        \Norm{1 \vee \nu^{\om}}{q, q'\!,Q(n)}
        }{(\si-\si')^2}
    \bigg)^{\!\!\ka}\,
    \Norm{(u - h)_+}{2p_*, 2p_*', Q_{\si}(n)}. \nonumber
  \end{align}
\end{theorem}
The proof of Theorem~\ref{thm:max:ineq} relies on the following two lemmas, an interpolation inequality for time-space averaged norms and an energy estimate for solutions of parabolic equation with time-dependent weights.  Let $Q = I \times B$ be a time-space cylinder, where $I = [s_1, s_2]$ is an interval and $B$ is a finite, connected subset of $V$. 
\begin{lemma}\label{lemma:interp}
 For any  $\rho > 1$ and $q' \in [1, \infty]$ let $\ga_1 \in (1,\rho]$ and $\ga_2 \in [q' / (q'+1), \infty)$ be such that
  \begin{align}\label{eq:interp:cond}
    \frac{1}{\ga_1} \,+\,
    \frac{1}{\ga_2}\, \bigg(1 - \frac{1}{\rho}\bigg) \frac{q'}{q'+1}
    \;=\;
    1.
  \end{align}
  Then, for any $u\!: I \times B \to \bbR$, 
  \begin{align}\label{eq:interp}
    \Norm{u}{\ga_1, \ga_2, Q}
    \;\leq\;
    \Norm{u}{1, \infty, Q} \,+\, \Norm{u}{\rho, q'/(q'+1), Q}.
  \end{align}
\end{lemma}
\begin{proof}
  This follows by an application of H\"older's and Young's inequality, as in \cite[Lemma~1.1]{KK77}
\end{proof}
\begin{lemma}\label{lemma:energy:est}
  Consider a smooth function $\ze\!: \bbR \to [0, 1]$ with $\ze = 0$ on $(-\infty,s_1]$ and a function $\eta\!: V \to [0, 1]$ such that
  \begin{align*}
    \supp \eta \;\subset\; B
    \qquad\text{and} \qquad
    \eta \;\equiv\; 0 \quad \text{on} \quad \partial B.
  \end{align*}
  Suppose that $\partial_t u - \cL_t^{\om}\, u \leq 0$ on $Q$.  Then, for any $k \geq 0$ and $p, p'\in (1, \infty)$, 
  \begin{align}\label{eq:energy:est}
    &\frac{1}{|I|}\, \Norm{\ze \eta^2 (u-k)_+^2}{1, \infty, Q}
    \,+\,
    \frac{1}{|I|}
    \int_{I} \ze(t)\; \frac{\cE_{t}^{\om}(\eta (u_t-k)_+)}{|B|} \, \md t
    \nonumber\\[1ex]
    &\mspace{72mu}\leq\;
    4\, \Norm{1 \vee \mu^{\om}}{p, p', Q}\,
    \Big(
      \norm{\nabla \eta}{\infty}{E}^2
      +
      \|\ze'\|_{\raisebox{-.0ex}{$\scriptstyle L^{\raisebox{.2ex}{$\scriptscriptstyle {\!\infty\!}$}} (I)$}}
    \Big)\, \Norm{(u-k)_+^2}{p_*, p_*', Q},
  \end{align}
 with $p_* \ldef p/(p-1)$ and $p_*' \ldef p'/(p'-1)$. 
\end{lemma}
\begin{proof}
  Let $k \geq 0$ and consider a function $u$ such that $\partial_t u \leq \cL_t^{\om}\, u$ on $Q = I \times B$.  To lighten notation, we set $v = (u-k)_+$.  By using the discrete version of the product rule \eqref{eq:rule:prod}, we obtain for any fixed $t \in (s_1, s_2)$ that
  \begin{align}\label{eq:energy:est:1}
    \scpr{\nabla (\eta v_t)}{\om_t \nabla (\eta v_t)}{E}
    &\;\leq\;
    \scpr{\nabla (\eta^2 v_t)}{\om_t \nabla v_t}{E}
    \,+\, \scpr{\av{v_t}^2}{\om_t (\nabla \eta)^2}{E},
  \end{align}
  where we used that $\av{\eta}^2(e) \leq \av{\eta^2}(e)$ by Jensen's inequality.  Further, by distinguishing four cases it follows that $\nabla(\eta^2 v_t)(e) (\nabla v_t)(e) \leq \nabla (\eta^2 v_t)(e) (\nabla u_t)(e)$ for any $e \in E$.  Hence,
  \begin{align*}
    \scpr{\nabla (\eta^2 v_t)}{\om_t \nabla v_t}{E}
    \;\leq\;
    \scpr{\nabla (\eta^2 v_t)}{\om_t \nabla u_t}{E}
    \;\leq\;
    \scpr{\eta^2 v_t}{-\partial_t u_t}{V}.
  \end{align*}
  Since the map $s \mapsto (s-k)_+$ is continuous on $\bbR$ with piecewise continuous and bounded derivative, we get that $\partial_t (u_t-k)_+^2 = 2 (u_t-k)_+\, \partial_t u_t$.  In particular,
  \begin{align}\label{eq:energy:est:2}
    \scpr{\nabla (\eta^2 v_t)}{\om_t \nabla v_t}{E}
    \;\leq\;
    -\frac{1}{2}\partial_t \scpr{\eta^2\!}{v_t^2}{V}.
  \end{align}
  By combining \eqref{eq:energy:est:1} and \eqref{eq:energy:est:2}, we deduced that
  \begin{align}\label{eq:energy:est:3}
    \partial_t \Norm{\eta^2 v_t^2}{1, B}
    \,+\,
    \frac{\cE_{t}^{\om}\big(\eta v_t\big)}{|B|}
    &\;\leq\;
    2\norm{\nabla \eta}{\infty}{E}^2\,
    \Norm{v_t^2 \mu_t^{\om}}{1, B}.
  \end{align}
  Moreover, since $\ze(s_1) = 0$,
  \begin{align*}
    \int_{s_1}^{s}\!\!
      \ze(t)\, \partial_t \Norm{\eta^2 v_t^2}{1, B}\;
    \md t
    &\;=\;
    \int_{s_1}^{s}
      \Big(
        \partial_t\big( \ze(t) \Norm{\eta^2 v_t^2}{1, B} \big)
        -
        \ze'(t) \Norm{\eta^2 v_t^2}{1,B}
      \Big)\;
    \md t
    \\[.5ex]
    &\;\geq\;
    \ze(s) \Norm{\eta^2 v_s^2}{1, B}
    \,-\,
    \|\ze'\|_{\raisebox{-.0ex}{$\scriptstyle L^{\raisebox{.2ex}{$\scriptscriptstyle {\!\infty\!}$}} (I)$}}
    \, \int_{s_1}^{s_2} \Norm{v_t^2}{1, B}\; \md t
  \end{align*}
  for any $s \in (s_1, s_2]$.  Thus, by multiplying both sides of \eqref{eq:energy:est:3} with $\ze$ and integrating the resulting inequality over $[s_1, s]$ for any $s \in I$, the assertion \eqref{eq:energy:est} follows by an application of H\"older's and Jensen's inequality.
\end{proof}
\begin{proof}[Proof of Theorem~\ref{thm:max:ineq}]
  For any $p, p' \in (1, \infty)$, let $p_* \ldef p/(p-1)$ and $p_*' \ldef p'/(p'-1)$ be the H\"older conjugate of $p$ and $p'$, respectively, and set
  \begin{align}\label{eq:def:alpha}
    \al
    \;\ldef\;
    \frac{1}{p_*}
    \,+\,
    \frac{1}{p_*'}\bigg(1 - \frac{1}{\rho} \bigg)\, \frac{q'}{q'+1},
  \end{align}
  where $\rho$ is defined in \eqref{eq:def:rho}.  Notice that for any $p, p', q, q' \in (1,\infty]$ for which \eqref{eq:cond:pq} is satisfied, $\al > 1$ and therefore $1/\al_* = 1 - 1/\al > 0$.  In particular, $\al > 1$ implies that $\al p_*' > q'/(q'+1)$ and $\al p_* \leq \rho$ so that Lemma~\ref{lemma:interp} is applicable.  Suppose that $n \geq 2(N_1(x_0) \vee N_2(x_0))$.  The remaining part of the proof comprises two steps, a ``one-step estimate'' and the iteration scheme.

  \textsc{Step~1.} Let $1/2 \leq \si' < \si \leq 1$ and $0 \leq k < l$ be fixed constants.  We write $I_{\si} \ldef [t_0 - \si n^2, t_0]$, $B_{\si} \ldef B(x_0, \si n)$ and $Q_{\si} \ldef I_{\si} \times B_{\si}$ to simplify notation.  Note that $|I_{\si}|/|I_{\si'}| \leq 2$ and $|B_{\si}| / |B_{\si'}| \leq 2^d C_{\mathrm{reg}}/c_{\mathrm{reg}}$.  Let us stress the fact, that, due to the discrete structure of the underlying space, the discrete balls $B_{\si'}$ and $B_{\si}$ may coincide even if $\si' < \si$.  In order to ensure that $B_{\si'} \subsetneq B_{\si}$, we assume in the sequel that $(\si - \si') n \geq 1$.  In this case, we can define a cut-off function $\eta\!:V \to [0,1]$ in space having the properties that $\supp \eta \subset B_{\si}$, $\eta \equiv 1$ on $B_{\si'}$, $\eta \equiv 0$ on $\partial B_{\si}$ and $\norm{\nabla \eta}{\infty}{E} \leq 1/((\si - \si') n)$.  Moreover, let $\ze\!:\bbR \to [0,1]$ be a smooth cut-off function in time such that $\ze \equiv 1$ on $I_{\si'}$, $\ze \equiv 0$ on $(-\infty, t_0 - \si n^2]$ and $\| \ze' \|_{\raisebox{-0ex}{$\scriptstyle L^{\raisebox{.1ex}{$\scriptscriptstyle \!\infty$}} (\bbR)$}} \leq  1 / ((\si - \si') n^2)$. 

  The constant $c \in (0, \infty)$ appearing in the computations below is independent of $n$ but may change from line to line.  First, by H\"older's inequality, we get
  \begin{align}
    &\Norm{(u-l)_+^2}{p_*, p_*', Q_{\si'}}
    \nonumber\\[1ex]
    &\mspace{36mu}\leq\;
    \Norm{(u-k)_+^2}{\al p_*, \al p_*', Q_{\si'}}\,
    \Norm{\indicator_{\{u \,\geq\, l\}}}{\al_* p_*, \al_* p_*', Q_{\si'}}
    \nonumber\\[1ex]
    &\mspace{31mu}\overset{\!\eqref{eq:interp}\!}{\;\leq\;}
    \Big(
    \Norm{(u-k)_+^2}{1, \infty, Q_{\si'}} \,+\,
    \Norm{(u-k)_+^2}{\rho,q'/(q'+1),Q_{\si'}}
    \Big)\,
    \Norm{\indicator_{\{u \,\geq\, l\}}}{p_*, p_*', Q_{\si'}}^{1/\al_*}.
  \end{align}
  Then, the local space-time Sobolev inequality and the energy estimate yields
  \begin{align*}
    \Norm{(u-k)_+^2}{\rho,q'/(q'+1),Q_{\si'}}
    &\;\leq\;
    c\, \Norm{\ze \eta^2(u-k)_+^2}{\rho,q'/(q'+1),Q_{\si}}
    \\[.5ex]
    &\;\overset{\!\!\eqref{eq:sob:ineq:weight}\!\!}{\leq\;}
    c\, n^2\, \Norm{1 \vee \nu^{\om}}{q,q',Q_{\si}}
    \bigg(
      \frac{1}{|I_{\si}|}
      \int_{I_{\si}}\!
        \ze(t)\, \frac{\cE_t^{\om}(\eta(u_t-k)_+)}{|B_{\si}|}\;
      \md t
    \bigg)
    \\[.5ex]
    &\;\overset{\!\!\eqref{eq:energy:est}\!\!}{\leq\;}
    c\;
    \frac{
      \Norm{1 \vee \mu^{\om}}{p, p', Q_{\si}}
      \Norm{1 \vee \nu^{\om}}{q,q',Q_{\si}}
    }{(\si-\si')^{2}}\;
    \Norm{(u-k)_+^2}{p_*, p_*', Q_{\si}}
  \end{align*}
  and
  \begin{align}
    \Norm{(u-k)_+^2}{1, \infty, Q_{\si'}}
    &\;\leq\;
    c\, \Norm{\ze \eta^2(u-k)_+^2}{1, \infty, Q_{\si}}
    \nonumber\\[.5ex]
    &\;\overset{\!\!\eqref{eq:energy:est}\!\!}{\leq\;}
    c\,
    \frac{\Norm{1 \vee \mu^{\om}}{p, p',Q_{\si}}}
    {(\si - \si')^2}\;
    \Norm{(u-k)_+^2}{p_*, p_*', Q_{\si}}.%\mspace{88mu}
    \label{eq:maximal:estimate:time}
  \end{align}
  By combining the estimates above and using the fact that
  \begin{align*}
    \Norm{\indicator_{\{u \,\geq\, l\}}}{p_*, p_*', Q_{\si'}}
    \;\leq\;
    c\, \Norm{\indicator_{\{u-k \,\geq\, l-k\}}}{p_*, p_*', Q_{\si}}
    \;\leq\;
    \frac{c}{(l-k)^2}\, \Norm{(u-k)_+^2}{p_*, p_*', Q_{\si}},
  \end{align*}
  we finally obtain that
  \begin{align*}
    \Norm{(u-l)_+^2}{p_*, p_*', Q_{\si'}}
    \;\leq\;
    c\,
    \frac{\Norm{1 \vee \mu^{\om}}{p, p', Q(n)}\,
      \Norm{1 \vee \nu^{\om}}{q,q',Q(n)}}
    {(l-k)^{2/\al_*}(\si-\si')^2}\,
    \Norm{(u-k)_+^2}{p_*, p_*', Q_{\si}}^{1+1/\al_*}.
  \end{align*}
  Set $\vp(l, \si') \ldef \Norm{(u-l)_+^2}{p_*, p_*', Q_{\si'}}$ and $M \ldef c\, \Norm{1 \vee \mu^{\om}}{p, p', Q(n)}\, \Norm{1 \vee \nu^{\om}}{q,q',Q(n)}$.  Then, the inequality above reads
  \begin{align}\label{eq:iteration:max:ineq}
    \vp(l, \si')
    \;\leq\;
    \frac{M}{(l-k)^{2/\al_*} (\si-\si')^2}\,
    \vp(k, \si)^{1+1/\al_*}.
  \end{align}
  Note that the function $[0,\infty) \ni k \mapsto \vp(k, \si)$ is non-increasing for any $\si \in [1/2, 1]$.

  \textsc{Step~2.} Suppose that $n \geq 2(N_1(x_0) \vee N_2(x_0))$.  Let $h \geq 0$ be arbitrary but fixed and, for any $\De \in [0, 2/(d+2) )$, suppose that $1/2 \leq \si' < \si \leq 1$ are chosen in such a way that $\si-\si' > 4n^{-\De}$.  Further, for $j \in \bbN_0$ define 
  \begin{align*}
    \si_j \;\ldef\; \si' + 2^{-j}(\si-\si'),
    \qquad 
    k_j \ldef h + K(1-2^{-j})
  \end{align*}
  with $K \ldef 2^{(1+\al_*)^2}(M/(\si - \si')^{2})^{\al_*/2} \vp(h, \si)^{1/2}$. Set $J \ldef \lceil (d \ln n) / (2 \al_* \ln 2) \rceil$.  Obviously, $J \geq 1$.  Since $\al_* \geq (d+2)/2$, it follows that
  \begin{align*}
    (\si_{j-1} - \si_j) n
    \;=\;
    2^{-j} (\si - \si') n
    \;>\;
    2 n^{1-\De - d/(2\al_*)}
    \;\geq\;
    2,
    \qquad \forall\, j = 1, \ldots, J.
  \end{align*}
  We claim that, by induction,
  \begin{align}\label{eq:iteration:claim}
    \vp(k_j, \si_j)
    \;\leq\;
    \frac{\vp(h, \si)}{r^j}
    \qquad \forall\, j \in \{0, \ldots, J\}
  \end{align}
  where $r = 2^{2(1+\al_*)}$.  Indeed, for $j = 0$ the assertion is obvious.  Now suppose that \eqref{eq:iteration:claim} holds for any $j \in \{0, \ldots, J-1\}$. Then, in view of \eqref{eq:iteration:max:ineq}, we obtain
  \begin{align*}
    \vp(k_{j+1}, \si_{j+1})
    &\overset{\mspace{-7mu}\eqref{eq:iteration:max:ineq}\mspace{-7mu}}{\;\leq\;}
    M \bigg(\frac{2^{j+1}}{K}\bigg)^{\!\!2/\al_*}\,
    \bigg(\frac{2^{j+1}}{\si-\si'}\bigg)^{\!\!2}\,
    \vp(k_{j}, \si_{j})^{1+1/\al_*}
    \\[.5ex]
    &\;\leq\;
%    M\, \frac{2^{(\al+\be)(j+1)}}{K^{\al}\, (\si-\si')^{\be}}
    M \bigg(\frac{2^{j+1}}{K}\bigg)^{\!\!2/\al_*}\,
    \bigg(\frac{2^{j+1}}{\si-\si'}\bigg)^{\!\!2}\,
    \bigg(\frac{\vp(h, \si)}{r^{j}}\bigg)^{\!\!1+1/\al_*}
    \;\leq\;
    \frac{\vp(h, \si)}{r^{j+1}},  
  \end{align*}    
  which completes the proof of the claim \eqref{eq:iteration:claim}.  Moreover, since $(n^d 2^{2J}) / r^{J} \leq 1$ and $(\si_{J-1} - \si_{J})n \geq 1$, we obtain that
  \begin{align*}
    \max_{(t,x) \in Q_{\si_{\!J}}} (u(t,x) - k_{J})_+^2
    % \Norm{(u-k_J)_+^2}{\infty, \infty, Q_J}
    &\;\leq\;
    c \, n^d\, \Norm{(u-k_{J})_+^2}{1, \infty, Q_{\si_{\!J}}}
    \\[.5ex]
    &\overset{\mspace{-7mu}\eqref{eq:maximal:estimate:time}\mspace{-7mu}}
    {\;\leq\;}
    c\, n^d 2^{2J}\,
    \frac{\Norm{1 \vee \mu^{\om}}{p, p', Q(n)}}{(\si - \si)^2}\,
    \vp(k_{J-1}, \si_{J-1})
    \\[.5ex]
    &\overset{\mspace{-7mu}\eqref{eq:iteration:claim}\mspace{-7mu}}{\;\leq\;}
    c\, \frac{\Norm{1 \vee \mu^{\om}}{p, p', Q(n)}}{(\si - \si)^2}\,
    \vp(h, \si).
  \end{align*}
  Hence,
  \begin{align*}
    \max_{(t,x) \in Q_{\si'}(n)} u(t,x)
    %\\[.5ex]  
    %&\mspace{36mu}\leq\;
    \;\leq\;   
    h + K + c\,
    \bigg(
      \frac{\Norm{1 \vee \mu^{\om}}{p, p', Q(n)}}{(\si - \si)^2}
    \bigg)^{\!\!1/2}\,
    \Norm{(u-h)_+}{2p_*, 2p'_*,Q_{\si}(n)},
    % \\[.5ex]
    % &\mspace{36mu}\leq\;
    % h + C_1\,
    % \bigg(
    %   \frac{
    %     \Norm{1 \vee \mu^{\om}}{p, p', Q(n)}\,
    %     \Norm{1 \vee \nu^{\om}}{q, q', Q(n)}
    %   }
    %   {(\si - \si)^2}
    % \bigg)^{\!\!\al_*/2}\,
    % \Norm{(u-h)_+}{2p_*, 2p'_*,Q_{\si}(n)},
  \end{align*}
  and the assertion \eqref{eq:max:ineq} follows with $\ka \ldef \al_*/2$.
\end{proof}
As an application of Theorem~\ref{thm:max:ineq} we derive a near-diagonal bound for the heat kernel, which we now introduce.  For $s \in \bbR$ and $x \in V$, let $\Prob_{s, x}^{\om}$ be the probability measure on the space of $V$-valued c\`adl\`ag functions on $\bbR$, under which the coordinate process $(X_t : t \in \bbR)$ is the continuous-time Markov chain on $V$ starting at time $s$ in $x$ with time-dependent generator $\cL_t^{\om}$ as defined in \eqref{eq:def:generator}.  
Recall that, for any $x, y \in V$ and $s, t \in \bbR$ with $t \geq s$ we denote by $p^{\om}(s,x;t,y)$ the transition density (or heat kernel associated to $\cL_t^{\om}$), that is $p^{\om}(s,x;t,y) = \Prob^{\om}_{s, x}[X_t = y]$.  Note that the Markov property implies immediately that, for any $s \in \bbR$ and $x \in V$, the function $(t,y) \mapsto u_t(y) \ldef p^{\om}(s,x;s+t,y)$ solves
\begin{align*}
  \partial_t u_t(y) \;=\; \cL_t^{\om} u_t (y),
  \qquad \forall\, t > 0,  \quad y \in V.  
\end{align*}
\begin{corro}[Near-diagonal heat kernel upper bound]  \label{cor:hk_upper}
  Suppose Assumption~\ref{ass:graph}-(i) and (ii) hold.  Then, for any $x_1, x_2 \in V$, $s \geq 0$ and $p, p', q, q' \in [1, \infty]$ satisfying \eqref{eq:cond:pq}, there exist  $\ka' \equiv \ka'(p, p', q, q', d') \in (0, \infty)$, $C_3 \equiv C_3(p, p', q, q', d') < \infty$ and $N_6(x_2)<\infty$ such that for all $\sqrt{t} \geq N_6(x_2)$ and $y \in B(x_2, \frac{1}{2}\sqrt{t})$,
  \begin{align}\label{eq:hk:upper_bound}
    p^{\om}(s,x_1;s+t,y)
    \;\leq\;
    C_3\,
    \Big(
      \Norm{1 \vee \mu^{\om}}{p, p'\!, Q}\,
      \Norm{1 \vee \nu^{\om}}{q, q'\!, Q}
    \Big)^{\!\!\ka'}\,t^{-d/2},
  \end{align}
  where $Q = [0, t] \times B(x_2, \sqrt{t})$.
\end{corro}
\begin{proof}
  We wish to apply the maximal inequality of Theorem~\ref{thm:max:ineq} iteratively to the function $(t,y) \mapsto u_t(y) \equiv p^{\om}(s, x_1;s+t,y)$ on the cylinder $Q(n) \equiv Q(t_0, x_0, n)$ with $t_0 = t$, $x_0 = x_1$ and $n = \sqrt{t}$. We fix $\si' \ldef 1/2$, $\si \ldef 1$, and set $\si_j = \si - 2^{-j}(\si -\si')$ for any $j \in \bbN_0$.  Note that $\si_j \uparrow \si$ and $\si_0 = \si'$. By H\"older's inequality we have
  \begin{align*}
    \Norm{u}{2\al p_*, 2\al p'_*, Q(n)}
    \;\leq\;
    \Norm{u}{1,1,Q(n)}^{\ga}\, \Norm{u}{\infty,\infty,Q(n)}^{1-\ga},
  \end{align*}    
  with $\ga = 1/\max\{2\al p_*, 2\al p'_*\}$ and $\al$ as defined in \eqref{eq:def:alpha}.  Set $J \ldef \lfloor (\ln n) / (4d \ln 2) \rfloor$ and $\De \ldef 1/2d < 2/(d+2)$.  Then, for all $n \geq N_4(x_2) \ldef 2(N_1(x_2) \vee N_2(x_2)) \vee 2^{12d}$ it holds that $J \geq 3$ and $(\si_{j}-\si_{j-1}) > 4 n^{-\De}$ for all $j \in \{1, \ldots, J\}$.  Hence, an application of the maximal inequality~\eqref{eq:max:ineq} yields
  \begin{align*}
    \Norm{u}{\infty, \infty, Q_{\si_{j-1}}(n)}
    \;\leq\;
    2^{2\ka j} K\, \Norm{u}{1,1,Q(n)}^{\ga}\,
    \Norm{u}{\infty,\infty,Q_{\si_j}(n)}^{1-\ga}
    \qquad \forall\, j \in \{1, \ldots, J\},
  \end{align*}
  where we introduced $K = c\,\big(\Norm{1 \vee \mu^{\om}}{p, p'\!, Q(n)}\, \Norm{1 \vee \nu^{\om}}{q, q'\!,Q(n)}\big)^{\ka}$ to simplify the notation.  By iterating the inequality above, we get
  \begin{align*}
    \Norm{u}{\infty, \infty, Q_{1/2}(n)}
    &\;\leq\;
    2^{2\ka \sum_{j=0}^{J-1} (j+1)(1 - \ga)^j}\,
    \Big(
      K\, \Norm{u}{1, 1, Q(n)}^{\ga}
    \Big)^{\!\sum_{j=0}^{J-1} (1 - \ga)^j}
    \Norm{u}{\infty, \infty, Q_{\si_{\!J}}(n)}^{\ga (1 - \ga)^J}
    \\[.5ex]
    &\;\leq\;
    2^{2\ka / \ga^2}\, K^{1/\ga}\, \Norm{u}{1, 1, Q(n)}^{1 - (1-\ga)^J}\,
    \Norm{u}{\infty, \infty, Q_{\si_{\!J}}(n)}^{ (1 - \ga)^J}.
  \end{align*}
  Further, note that $u_t(y) = p^{\om}(s,x_1;s+t,y) \leq 1$ for all $t \geq 0$ and $y \in V$,
  \begin{align*}
    \sum_{y \in B(x_2, n)}\mspace{-12mu} u_t(y)
    \;=\;
    \Prob_{s, x_1}^{\om}\!\big[X_t \in B(x_2,n)\big]
    \;\leq\;
    1,
    \qquad \forall\, t \geq 0,
  \end{align*}
  and $|B(x_1, n)|^{(1-\ga)^{J}} \leq c < \infty$ uniformly in $n$.  Hence, by using the volume regularity, we conclude that
  \begin{align*}
    \Norm{u}{\infty, \infty, Q_{1/2}(n)}
    \;\leq\;
    c\, K^{1/\ga}\, |B(x_2, n)|^{(1-\ga)^J - 1}
    \;\leq\;
    c\, K^{1/\ga}\, n^{-d}.
  \end{align*}
  Since $(t,y) \in Q_{1/2}(t, x_1, \sqrt{t})$ for any $y \in B(x_1, \frac{1}{2} \sqrt{t})$, the assertion follows.
\end{proof}

\subsection{Proof of the oscillation bound} \label{sec:proof_osc}
In this subsection we prove Theorem~\ref{thm:osc}.  Inspired by the strategy that has been used in \cite{WYW06} to prove H\"older regularity for parabolic equations, we start by constructing a continuously differentiable version of the function $(0, \infty) \ni r \mapsto (-\ln r)_+$.  Consider the function $g\!: (0, \infty) \to [0, \infty)$,
\begin{align}\label{eq:def:g}
  g(r)
  \;\ldef\;
  \begin{cases}
    -\ln r, \quad &r \in (0, c_*],
    \\[.5ex]
    \dfrac{(r-1)^2}{2 c_* (1-c_*)},
    \quad &r \in (c_*, 1],
    \\[.5ex]
    0, &r \in (1, \infty),
  \end{cases}
\end{align}
where $c_* \in [1/4, 1/3]$ is the smallest solution of the equation $2 c \ln(1/c) = 1-c$.  By construction, the function $g$ is convex, non-increasing and in $C^1((0, \infty))$.
\begin{lemma}
  Suppose that $u > 0$ satisfies $\partial_t u - \cL_t^{\om}\, u \geq 0$ on $Q$, and let $g$ be the function  defined in \eqref{eq:def:g}.   Further, consider a cut-off function $\eta\!:V \to [0, 1]$  with
  \begin{align*}
    \supp \eta \;\subset\; B
    \qquad\text{and} \qquad
    \eta \;\equiv\; 0 \quad \text{on } \partial B.
  \end{align*}
  Then,
  \begin{align}\label{eq:energy:ln}
    \partial_t \Norm{\eta^2 g(u_t)}{1, B}
    \,+\,
    \frac{\cE_{t}^{\om, \eta^2}\!\big(g(u_t)\big)}{6 |B|}
    \;\leq\;
    6\, \Norm{1 \vee \mu_t^{\om}}{1, B}\,
    \osr(\eta)^2\, \norm{\nabla \eta}{\infty}{E}^2,
  \end{align}
  where $\osr(\eta) \ldef \max\big\{\eta(y) / \eta(x) \vee 1 \mid \{x,y\} \in E,\; \eta(x) \ne 0\big\}$ and
  \begin{align*}
    \cE_t^{\om, \eta^2}\!(f)
    \;\ldef\;
    \sum_{e \in E} \big(\eta^2(e^+) \wedge \eta^2(e^-)\big) \om(e)\,
    ( \nabla f)(e)^2.
  \end{align*}
\end{lemma}
\begin{proof}
  Since $\partial_t u - \cL_t^{\om} u \geq 0$ on $Q = I \times B$ and $u > 0$, we have
  \begin{align*}
    \partial_t \scpr{\eta^2}{g(u_t)}{V}
    &\;=\;
    \scpr{\eta^2 g'(u_t)}{\partial_t u_t}{V}
    \\[.5ex]
    &\;\leq\;
    \scpr{\eta^2 g'(u_t)}{\cL_t^{\om} u_t}{V}
    \;=\;
    -\scpr{\nabla (\eta^2 g'(u_t))}{\om_t \nabla u_t}{E}.
  \end{align*}
  Notice that $g'$ is piecewise differentiable and $1/3 g'(r)^2 \leq g''(r)$ for a.e.\ $r \in (0, \infty)$.  In particular, $-r g'(r) \leq 4/3$ for any $r \in (0, \infty)$.  Hence, by Lemma~\ref{lem:A1} we get
  \begin{align*}
    -\scpr{\nabla (\eta^2 g'(u_t))}{\om_t \nabla u_t}{E}
    \;\leq\;
    -\frac{1}{6}\cE^{\om, \eta^2}\!\big(g(u_t)\big)
    \,+\, 6\, \osr(\eta)^2\, \scpr{\nabla \eta}{\om_t \nabla \eta}{E}.
  \end{align*}
  Thus, by combining the estimates above and exploiting the fact that $g \geq 0$, the assertion \eqref{eq:energy:ln} follows.
\end{proof}
In the next lemma we show for a space-time harmonic function $u$ that, if the size of its sub-level set with respect to some $k_0$ is bounded from below by a fraction of the size of the time-space cylinder, then the size of the sub-level sets for fixed $t$ and a possibly larger constants, $k_j$, are bounded from below by a fraction of $B(n)$, provided that $t$ is close to $t_0$.  For that purpose, set for some $k_0 \in (-\infty, M_n)$,
\begin{align}\label{eq:def:k_j}
  k_j
  \;\ldef\;
  M_n - 2^{-j}\,\big(M_n - k_0\big),
  \qquad  j \in \bbN_0,
\end{align}
where $M_n \ldef \sup_{(t,x) \in Q(n)} u(t,x)$.  For $\eta\!: V \to [0, 1]$ with $\supp \eta \subset B(x_0, n)$ we write
\begin{align*}
  \Norm{u}{1, B(n), \eta^2}
  \;\ldef\;
  \frac{1}{\scpr{\eta^2}{1}{V}}\, \sum_{x \in B(x_0, n)} \eta^2(x)\, |u(x)|
\end{align*}
to denote the $\eta^2$-weighted $\ell^1$-norm of a function $u\!: V \to \bbR$.
\begin{lemma}\label{lemma:levelset:apriori}
  Let Assumption~\ref{ass:graph}-(i) be satisfied and $\si\in (0,1)$ be fixed.  For $t_0 \in \bbR$, $x_0 \in V$ and $\si n \geq N_1(x_0)$, suppose that $u > 0$ is such that $\partial_t u - \cL_t^{\om} u = 0$ on $Q(n)$.  Further, let $\eta\!: V \to [0,1]$, $x \mapsto \eta(x) \ldef [1-d(x_0, x)/(\sigma n)]_+$ be a cut-off function in space and set $B(n) \equiv B(x_0, n)$.  If for some $k_0 \in (-\infty, M_n)$,
  \begin{align}\label{eq:apriori:estimate}
    \frac{1}{n^2}\,
    \int_{t_0-n^2}^{t_0}
      \Norm{\indicator_{\{u_t \leq k_0\}}}{1, B(n), \eta^2}\;
    \md t
    \;\geq\;
    \frac{1}{2},
  \end{align}
  then for $\tau = 1/4$ there exists $\de \in (0, 1/3)$ and $i \equiv i(\om,\sigma) < \infty$ such that for all $j \geq i$,
  \begin{align*}
    \Norm{\indicator_{\{u_t \leq k_j\}}}{1, B(\sigma n)}
    \;\geq\;
    \de,
    \qquad
    \forall\, t \in [t_0 - \tau n^2, t_0].
  \end{align*}
\end{lemma}
\begin{proof}
  In order to simplify the presentation, set
  \begin{align*}
    v_t(x)
    \;\ldef\;
    \frac{M_n - u_t(x)}{M_n - k_0}
    \qquad \text{and} \qquad
    h_j
    \;=\;
    \ve_j
    \;\ldef\;
    2^{-j}.
  \end{align*}
  Then, $\partial_t (v + \ve_j) - \cL_t^{\om} (v + \ve_j) = 0$ on $Q(n)$ for all $j \in \bbN_0$ and $u_t(x) > k_j$ if and only if $v_t(x) < h_j$ for any $x \in V$.  Set $\bar{\tau} \ldef 1/3$.  In view of \eqref{eq:apriori:estimate} there exists $s \in [t_0- n^2, t_0 - \bar{\tau} n^2]$ such that
  \begin{align}\label{eq:apriori:pointwise}
    \Norm{\indicator_{\{v_s < 1\}}}{1, B(n), \eta^2}
    \;\leq\;
    \frac{3}{4}.
  \end{align}
  Indeed, if we assume that the contrary is true, that is $\Norm{\indicator_{\{v_s < 1\}}}{1, B(n), \eta^2} > 3/4$ for all $s \in [t_0 - n^2, t_0 - \bar{\tau} n^2]$, then we find that
  \begin{align*}
    \frac{1}{2}
    &\overset{\!\!\!\eqref{eq:apriori:estimate}\!\!\!}{\;\geq\;}
    \frac{1}{n^2}\,
    \int_{t_0-n^2}^{t_0}
      \Norm{\indicator_{\{v_t < 1\}}}{1, B(n), \eta^2}\;
    \md t
    \;>\;
    \frac{1}{n^2}\,
    \int_{t_0-n^2}^{t_0- \bar{\tau} n^2}
      \frac{3}{4}\,
    \md t
    \;=\;
    \frac{1}{2},
  \end{align*}
  which leeds to a contradiction.  By integrating the energy estimate \eqref{eq:energy:ln} over the interval $[s,t]$ with $t \in [t_0 -\tau n^2, t_0]$ and using that $\norm{\nabla \eta}{\infty}{E} \leq 1/(\sigma n)$ and $\osr(\eta) \leq 2$, we obtain
  \begin{align*}
    \Norm{g(v_t + \ve_j)}{1, B(n), \eta^2}
    \;\leq\;
    \Norm{g(v_s + \ve_j)}{1, B(n), \eta^2}
    \,+\,
    c\, \Norm{1 \vee \mu_t^{\om}}{1, 1, Q(n)},
  \end{align*}
  where $c \equiv c(\sigma) \in (0, \infty)$ is a constant independent of $n$.  Since, by construction, $g$ is non-increasing and vanishes on $[1, \infty)$ we find that
  \begin{align*}
    \Norm{g(v_s + \ve_j)}{1, B(n), \eta^2}
    &\;\leq\;
    g(\ve_j)\, \Norm{\indicator_{\{v_s < 1\}}}{1, B(n), \eta^2}
    \overset{\eqref{eq:apriori:pointwise}}{\;\leq\;}
    g(\ve_j) \cdot \frac{3}{4}
  \end{align*}
  and
  \begin{align*}
    \Norm{g(v_t + \ve_j)}{1, B(n), \eta^2}
    \;\geq\;
    g(h_j + \ve_j)\, \Norm{\indicator_{\{v_t < h_j\}}}{1, B(n), \eta^2}.
  \end{align*}
  This yields, for any $j \geq 2$,
  \begin{align*}
    \Norm{\indicator_{\{v_t < h_j\}}}{1, B(n), \eta^2}
    &\;\leq\;
    \frac{g(\ve_j)}{g(h_j + \ve_j)} \cdot \frac{3}{4}
    \,+\,
    \frac{c}{g(h_j + \ve_j)}\, \Norm{1 \vee \mu_t^{\om}}{1, 1, Q(n)}
    \\[1ex]
    &\;\leq\;
    \bigg(1 + \frac{1}{j-1}\bigg)\, \frac{3}{4}
    \,+\,
    \frac{c}{j-1}\, \Norm{1 \vee \mu_t^{\om}}{1, 1, Q(n)}.
  \end{align*}
  Hence, there exists $i(\om,\sigma) \in \bbN$ such that $\Norm{\indicator_{\{v_t < h_j\}}}{1, B(n), \eta^2} \leq 5/6$ for all $j \geq i(\om,\sigma)$ and $t \in [t_0 -\tau n^2, t_0]$.
  Since $\scpr{\eta^2}{1}{V} \geq |B(\sigma n/2)|/4$, we obtain
  \begin{align*}
    \Norm{\indicator_{\{u_t \leq k_j\}}}{1, B(\sigma n)}
    % \;=\;
    % \Norm{\indicator_{\{v_t \geq h_j\}}}{1, B(\sigma n)}
    \;\geq\;
    \frac{|B(\sigma n/2)|}{4 |B(\sigma n)|}\,
    \Norm{\indicator_{\{v_t \geq h_j\}}}{1, B(n), \eta^2}
    \;\geq\;
    \frac{c_{\mathrm{reg}}}{6 \cdot 2^{d+2}\,C_{\mathrm{reg}}},
  \end{align*}
  and the assertion follows.
\end{proof}
Our next step is to show that the size of the sets where a space-time harmonic function $u$ exceeds some level $k$ can be made arbitrary small compared to the size of the time-space cylinder $Q(n)$ provided that $k$ is sufficiently close to the maximum of $u$ in $Q(n)$.
\begin{lemma}\label{lemma:levelset:small}
  Let Assumption~\ref{ass:graph}-(i) and (iii) be satisfied, and set $\tau \ldef 1/4$ and $\si \ldef 1/ (2 C_{\mathrm{W}})$.  For $t_0 \in \bbR$, $x_0 \in V$ and $\si n \geq N_1(x_0) \vee N_3(x_0)$, suppose that $u > 0$ solves $\partial u - \cL_t^{\om} u = 0$ on $Q(n)$.  Assume that there exists $\de > 0$ and $i \in \bbN$ such that
  \begin{align}\label{eq:ass:apriori}
    \Norm{\indicator_{\{u_t \leq k_i\}}}{1, B(x_0, \si n)}
    \;\geq\;
    \de
    \qquad \forall\, t \in [t_0 -\tau n^2, t_0].
  \end{align}
  Then, for any $\ve \in (0, 1)$ there exists $\bbN \ni j(\ve, \de, \om) \geq i$ such that
  \begin{align}
    \Norm{\indicator_{\{u > k_{j}\}}}{1, 1, Q_{\tau, \si}(n)}
    \;\leq\;
    \ve
    \qquad \forall\, j \geq j(\ve, \de, \om).
  \end{align}
\end{lemma}
\begin{proof}
  Write $I_{\tau} \ldef [t_0 - \tau n^2, t_0]$, $B_{\si} \ldef B(x_0, \si n)$ and $Q_{\tau, \si} \ldef I_{\tau} \times B_{\si}$.  Let $\eta\!: V \to [0, 1]$ be a cut-off function with the properties that $\supp \eta \subset B(n)$, $\eta \equiv 1$ on $B_{1/2}$, $\eta \equiv 0$ on $\partial B_1$ and linear decaying on $B_1 \setminus B_{1/2}$.  Thus, $\norm{\nabla \eta}{\infty}{E} \leq 2/ n$ and $\osr(\eta) \leq 2$.  Further, define
  \begin{align*}
    w_t(x)
    \;\ldef\;
    \frac{M_n - u_t(x)}{M_n - k_i}
    \qquad \text{and} \qquad
    h_j
    \;=\;
    \ve_j
    \;\ldef\;
    2^{-j}.
  \end{align*}
  Then, $w \geq 0$ and $\partial_t (w + \ve_j) - \cL_t^{\om} (w + \ve_j) = 0$ on $Q(n)$ for any $j \in \bbN$.  Define $\cN_t \ldef \{x \in B_{\si} : g(w_t(x) + \ve_j) = 0\}$ for any $t \in [t_0 - \tau n^2, t_0]$.  Recall that $g(r) = 0$ for every $r \in (1, \infty)$.  Hence,
  \begin{align*}
    \frac{|\cN_t|}{|B_{\si}|}
    \;=\;
    \Norm{\indicator_{\{g(w_t + \ve_j) = 0\}}}{1, B_{\si}}
    \;\geq\;
    \Norm{\indicator_{\{w_t \geq 1\}}}{1, B_{\si}}
    \;=\;
    \Norm{\indicator_{\{u_t \leq k_i\}}}{1, B_{\si}}
    \overset{\!\!\!\eqref{eq:ass:apriori}\!\!\!}{\;\geq\;}
    \de
    \;>\;
    0,
  \end{align*}
  and we deduce from Assumption~\ref{ass:graph}-(iii) that
  \begin{align*}
    \Norm{g(w_t + \ve_j)}{1, B_{\si}}
    \overset{\eqref{eq:poincare:ineq}}{\;\leq\;}
    C_{\mathrm{P}}\,n\, \Big(1 + \frac{1}{\de}\Big)\, \frac{|B_1|}{|B_{\si}|}\,
    \Norm{\nu_t^{\om}}{1, B_1}^{1/2}\,
    \bigg(
      \frac{\cE_t^{\om, \eta^2}\big(g(w_t + \ve_j)\big)}{|B_1|}
    \bigg)^{\!\!1/2}
  \end{align*}
  for any $t \in [t_0 - \tau n^2, t_0]$.  Hence,
  \begin{align}\label{eq:estimate:g:1norm}
    \Norm{g(w + \ve_j)}{1, 1, Q_{\tau, \si}}^2
    \;\leq\;
    \frac{c}{\de}\, \Norm{1 \vee \nu^{\om}}{1,1,Q(n)}\,
    \int_{I_{\tau}}
      \frac{\cE_t^{\om, \eta^2}(g(w_t + \ve_j))}{|B_1|}\,
    \md t,
  \end{align}
  where $c \in (0, \infty)$ is a constant independent of $n$ which may change from line to line.  An upper bound on the time-averaged Dirichlet form follows from the energy estimate.  Indeed, by integrating \eqref{eq:energy:ln} over the interval $[t_0 - \tau n^2, t_0]$ we obtain that
  \begin{align}\label{eq:estimate:g:df}
    \int_{I_{\tau}}
      \frac{\cE_t^{\om, \eta^2}(g(w_t + \ve_j))}{|B_1|}\,
    \md t
    \;\leq\;
    \frac{c}{\de}\,
    \Big(
      \Norm{g(w_{t_0 - \tau n^2} + \ve_j)}{1, B_1}
      \,+\,
      \Norm{1 \vee \mu^{\om}}{1, 1, Q(n)}
    \Big).
  \end{align}
  Thus, by combining \eqref{eq:estimate:g:1norm} and \eqref{eq:estimate:g:df} and using that $g$ is non-increasing we obtain that for any $j \geq 2$,
  \begin{align*}
    &\Norm{\indicator_{\{w < h_j\}}}{1, Q_{\tau, \si}}^2
    \\[0.5ex]
    &\mspace{36mu}\leq\;
    \frac{c}{\de}\, \Norm{1 \vee \nu^{\om}}{1,1,Q(n)}\,
    \bigg(
      \frac{g(\ve_j)}{g(h_j + \ve_j)^2}
      \,+\,
      \frac{1}{g(h_j + \ve_j)^2}\,
      \Norm{1 \vee \mu^{\om}}{1, 1, Q(n)}
    \bigg)
    \\[1ex]
    &\mspace{36mu}\leq\;
    \frac{c}{\de}\, \Norm{1 \vee \nu^{\om}}{1,1,Q(n)}\,
    \bigg(
      \frac{j}{(j-1)^2} \,+\, \frac{1}{(j+1)^2}\,
      \Norm{1 \vee \mu^{\om}}{1, 1, Q(n)}
    \bigg).
  \end{align*}
  Hence, for any $\ve > 0$ there exists $\bbN \ni j(\ve, \de, \om) \in [i, \infty)$ such that for all $j \geq j(\ve, \de, \om)$ it holds that $\Norm{\indicator_{\{u > k_{j}\}}}{1, 1, Q_{\tau, \si}} = \Norm{\indicator_{\{w < h_{j-i}\}}}{1, 1, Q_{\tau, \si}} \leq \ve$, which completes the proof.
\end{proof}
\begin{proof}[Proof of Theorem~\ref{thm:osc}]
  Obviously, if $M_n = m_n$, where
  \begin{align*}
    M_n
    \;\ldef\;
    \sup\nolimits_{(t,x) \in Q(n)} u(t,x)
    \qquad \text{and} \qquad
    m_n
    \;\ldef\;
    \inf\nolimits_{(t, x) \in Q(n)} u(t,x),
  \end{align*}
  then \eqref{eq:oscillation_ineq} holds true for any $\ga \geq 0$.  Therefore, we assume in what follows that $M_n > m_n$.  Set $k_0 \ldef ( M_n + m_n )/2 \in (-\infty, M_n)$, and let $k_j$ for any $j \in \bbN$ be defined as in \eqref{eq:def:k_j}.  Moreover, consider the cut-off function  $V \ni x \mapsto \eta(x) \ldef [1 - d(x_0, x)/n]_+$.  We may assume that
  \begin{align*}
    \frac{1}{n^2}\,
    \int_{t_0-n^2}^{t_0}
      \Norm{\indicator_{\{u_t \leq k_0}\}}{1, B(n), \eta^2}\;
    \md t
    \;\geq\;
    \frac{1}{2}.
  \end{align*}
  Otherwise, we consider $(M_n + m_n) - u$ instead of $u$.  Let $\tau = 1/4$, $\si = 1/(2C_{\mathrm{W}})$ and $\ve = 1 / \big(2^{2 \ka + 1} C_2 C_{\mathrm{W}}^2 (\Norm{1 \vee \mu^{\om}}{p, p', Q(\si n)} \Norm{1 \vee \nu^{\om}}{q, q', Q(\si n)})^{\ka}\big)^{p_* \!\vee p_*'}$.  Then, by applying Lemmas~\ref{lemma:levelset:apriori} and \ref{lemma:levelset:small} we find $j \equiv j(\om, \ve) < \infty$ such that
  \begin{align*}
    \Norm{\indicator_{\{u > k_{j}\}}}{1, 1, Q_{\tau, \si}(n)}
    \;\leq\;
    \ve.
  \end{align*}
  Next, set $\vt \ldef \si / 2$ and apply Theorem~\ref{thm:max:ineq} to obtain that
  \begin{align*}
    M_{\vt n}
    &\;\leq\;
    \sup_{(t,x) \in Q_{1/2}(\si n)}\mspace{-6mu} u(t,x)
    \\[.5ex]
    &\;\leq\;
    k_j
    \,+\,
    C_2\, 2^{2\ka}
    \Big(
      \Norm{1 \vee \mu^{\om}}{p, p'\!, Q(\si n)}\,
      \Norm{1 \vee \nu^{\om}}{q, q'\!,Q(\si n)}
    \Big)^{\!\ka}\,
    \Norm{(u - k_j)_+}{2p_*, 2p_*', Q(\si n)}.
  \end{align*}
  Since
  \begin{align*}
    \Norm{(u - k_j)_+}{2p_*, 2p_*', Q(\si n)}
%    &\;\leq\;
%    \big(M_n - k_j\big)\,
%    \Norm{\indicator_{\{u > k_j\}}}{2p_*, 2p_*', Q(\si n)}
%    \\[.5ex]
    &\;\leq\;
    C_{\mathrm{W}}^2\,
    \big(M_n - k_j\big)\,
    \Norm{\indicator_{\{u > k_{j}\}}}
      {1, 1, Q_{\tau, \si}(n)}^{1/p_* \wedge 1/p_*'},
  \end{align*}
  we find that
  \begin{align*}
    M_{\vt n}
    \;\leq\;
    k_j
    \,+\,
    \frac{1}{2}\, \big(M_n - k_j\big)
    \;=\;
    M_n \,-\, \frac{1}{2^{j+2}}\, \big(M_n - m_n\big).
  \end{align*}
  Therefore, we have
  \begin{align*}
    M_{\vt n} - m_{\vt n}
    \;\leq\;
    M_n \,-\, \frac{1}{2^{j+2}}\, \big(M_n - m_n\big) - m_{\vt n}
    \;\leq\;
    \Big(1 - \frac{1}{2^{j+2}}\Big)\, \big(M_n - m_n\big),
  \end{align*}
  which proves the theorem.
\end{proof}

\subsection{Weak Parabolic Harnack inequality}
As mentioned before, one appealing aspect of the above proof of Theorem~\ref{thm:osc} is its avoidance of a parabolic Harnack inequality (PHI). Nevertheless, from the maximal inequality in Theorem~\ref{thm:max:ineq} and the auxiliary estimates on the level sets of harmonic functions  in Lemmas~\ref{lemma:levelset:apriori} and \ref{lemma:levelset:small} we can deduce the following weaker version of the PHI.  In the continuous setting it also possible to derive a full PHI from the conjunction of a weak PHI and a H\"older continuity estimate, see e.g.\ \cite[Section~5.2.3]{WYW06}.  However, those arguments cannot be transferred into our discrete setting. 
\begin{theorem}\label{thm:weak:harnack}
  Suppose that Assumption~\ref{ass:graph} holds.  For $t_0 \in \bbR$, $x_0 \in V$ and $\si =1/(2C_{\mathrm{W}})$ fixed let $n\in \bbN$ be such that $\si n \geq \max\{N_1(x_0), N_2(x_0), N_3(x_0)\}$. Suppose that $u > 0$ is a solution of $\partial_t u - \cL_t^{\om} u = 0$ on $Q(n)$.  Further, let $\eta\!: V \to [0,1]$, $x \mapsto \eta(x) \ldef [1-d(x_0, x)/(\sigma n)]_+$ be a cut-off function in space and set $B(n) \equiv B(x_0, n)$.  If
  \begin{align}\label{eq:weak:Harnack:cond}
    \frac{1}{n^2}\,
    \int_{t_0-n^2}^{t_0}
      \Norm{\indicator_{\{u_t \geq h\}}}{1, B(n), \eta^2}\;
    \md t
    \;\geq\;
    \frac{1}{2},
  \end{align}
  for some $h>0$, then there exists $j \in \bbN$ such that
  \begin{align}\label{eq:weak:Harnack}
    \inf_{Q_{1/2}(\si n)} u(t,x) \;\geq\; \frac{h}{2^{j+1}}.
  \end{align}
\end{theorem}
\begin{proof}
  Define $v \ldef (M_n + m_n) - u$, where
  \begin{align*}
    M_n
    \;\ldef\;
    \sup\nolimits_{(t,x) \in Q(n)} u(t,x)
    \qquad \text{and} \qquad
    m_n
    \;\ldef\;
    \inf\nolimits_{(t, x) \in Q(n)} u(t,x).
  \end{align*}
  Note that, by definition, $v > 0$ and $\partial_t v - \cL_t^{\om} v = 0$ on $Q(n)$.  In particular, we have that $\sup_{(t,x) \in Q(n)} v(t,x) = M_n$ and $\inf_{(t,x) \in Q(n)} v(t,x) = m_n$.  Since, for any $h \in (0, m_n]$, the assertion \eqref{eq:weak:Harnack} is trivial, we assume in the sequel that $h > m_n$.  Set $k_0 \ldef M_n + m_n - h \in (-\infty, M_n)$.  Then, \eqref{eq:weak:Harnack:cond} is equivalent to
  \begin{align}
    \frac{1}{n^2}\,
    \int_{t_0-n^2}^{t_0}
      \Norm{\indicator_{\{v_t \leq k_0\}}}{1, B(n), \eta^2}\;
    \md t
    \;\geq\;
    \frac{1}{2}.
  \end{align}
  Let $\tau = 1/4$ and $\ve = 1 / \big(2^{2 \ka + 1} C_2 C_{\mathrm{W}}^2 (\Norm{1 \vee \mu^{\om}}{p, p', Q(\si n)} \Norm{1 \vee \nu^{\om}}{q, q', Q(\si n)})^{\ka}\big)^{p_* \!\vee p_*'}$.  Then, by applying Lemmas~\ref{lemma:levelset:apriori} and \ref{lemma:levelset:small} we find $j \equiv j(\om, \ve) < \infty$ such that
  \begin{align*}
    \Norm{\indicator_{\{v > k_{j}\}}}{1, 1, Q_{\tau, \si}(n)}
    \;\leq\;
    \ve,
  \end{align*}
  for $k_j$ as defined in \eqref{eq:def:k_j}.  Thus, by Theorem~\ref{thm:max:ineq} we obtain
  \begin{align*}
    &\sup_{(t,x) \in Q_{1/2}(\si n)}\mspace{-6mu} v(t,x)
    \\
    &\mspace{36mu}\leq\;
    k_j
    \,+\,
    C_2\, 2^{2\ka}
    \Big(
      \Norm{1 \vee \mu^{\om}}{p, p'\!, Q(\si n)}\,
      \Norm{1 \vee \nu^{\om}}{q, q'\!,Q(\si n)}
    \Big)^{\!\ka}\,
    \Norm{(v - k_j)_+}{2p_*, 2p_*', Q(\si n)}
    \\[.5ex]
    &\mspace{36mu}\leq\;
    k_j + \frac{1}{2}\, (M_n - k_j),  
  \end{align*}
  where we used that
  \begin{align*}
    \Norm{(v - k_j)_+}{2p_*, 2p_*', Q(\si n)}
    \;\leq\;
    C_{\mathrm{W}}^2\,
    \big(M_n - k_j\big)\,
    \Norm{\indicator_{\{v > k_{j}\}}}
      {1, 1, Q_{\tau, \si}(n)}^{1/p_* \wedge 1/p_*'}.
  \end{align*}
  By using the definition of $v$, we arrive that
  \begin{align*}
    \inf_{(t,x) \in Q_{1/2}(\si n)}\mspace{-6mu} u(t,x)
    \;\geq\;
    \frac{1}{2} (M_n - k_j) + m_n
    \;\geq\;
    \frac{h}{2^{j+1}},
  \end{align*}
  which is the claim.
\end{proof}

\section{A general criterion for a local CLT}\label{sec:localclt}
In this section we prove a local central limit theorem for suitably regular subgraphs $\cG\subset \bbZ^d$, provided that H\"older continuity on large space-time scales and a CLT hold. The proof will  mostly follow the arguments in the proof of a similar result in~\cite[Section 4]{BH09}, from which we borrow some of the notation. For further related results we refer to \cite{CH08}. However, the arguments in the present paper require a modification of the criteria in \cite{BH09,CH08}, which is why we state it here and also include a proof for the reader's convenience.

Let $\cG \subset \bbZ^d$  be an infinite and connected graph and let $d\!: \cG\times \cG \to [0,\infty)$ denote the graph distance on $\cG$. We assume that $0\in \cG$. For $x \in \bbR^d$ and $r > 0$ we set
\begin{align*}
  C(x,r) \;\ldef\; x + [- r,r]^d,
  \qquad
  \La(x,r) \;\ldef\;  C(x,r) \cap \cG,
  \qquad 
  \La_n(x,r) \;\ldef\; \La(nx, nr).
\end{align*}
Let $g_n\!: \bbR^d \to \cG$ be a function so that $g_n(x)$ is a closest point in $\cG$ to $nx$, in the $|\cdot|_\infty$ norm (we can put a fixed ordering on $\bbZ^d$ to resolve ties).  We denote by $Q$ the law of a time-continuous random walk $(X_t : t\geq 0)$ on $\cG$ started at $0 \in \cG$ at time $t = 0$.  We set 
\begin{align*}
  q(t,x) \;=\; Q[X_t = x],
  \qquad t \geq 0,\, x \in \cG.
\end{align*}
We assume the following additional properties on $\cG$ and $Q$.
\begin{itemize}
\item[($\cG.1$)] There exists $ \cC_\cG>0$  such that for any $r>0$ and $x\in \bbR^d$,
  \begin{align*}
    \frac{|\La_n(x,r)|}{(2n r)^d} \longrightarrow \cC_{\cG},
    \qquad \text{as } n \to \infty.
  \end{align*}
  
\item[($\cG.2$)] There exist $\de > 0$, a constant $C_4>0$ and $n_0 \in \bbN$ such that, for each $r > 0$ and all $n \geq n_0$,
  \begin{align}\label{eq:distance_bound}
    d(y,z) \;\leq\;  \big(C_4 \, |y - z|_\infty\big) \vee n^{1-\de},
    \qquad \forall\, y, z \in \La_n(x,r).
  \end{align}
  
\item[($\cG.3$)] There exists a symmetric matrix $\Si \in \bbR^{d \times d}$ such that for any $x \in \bbR^d$, $t > 0$ and $r > 0$,
  \begin{align*}
    Q\big[ n^{-1} X_{n^2 t} \in C(x,r) \big]
    \underset{n \to \infty}{\;\longrightarrow\;}
    \int_{C(x,r)} k^\Si_t(y) \, \mathrm{d} y,
  \end{align*}
  where $k_t^{\Si}$, defined in \eqref{eq:def_kt}, denotes the transition kernel of the Brownian motion with covariance $\Si^2$.

\item[($\cG.4$)] There exist $C_5 > 0$ and $\vr > 0$ such that for any $\de \in(0, 1)$, $\sqrt{t}/2 \geq \de$ and $x \in \bbR^d$,
  \begin{align*}
    \limsup_{n \to \infty}
    \sup_{ \substack{x_1, x_2 \in B(g_n(x), \de n) \\ t-\de^2 <s_1,s_2 \leq t}}
    n^d \, \big| q(n^2 s_1, x_1) - q(n^2 s_2, x_2) \big|
    \;\leq\;
    C_5\, \Big(\frac{\de}{\sqrt{t}}\Big)^{\!\vr} \, t^{-\frac{d}{2}}.
  \end{align*}
\end{itemize}
\begin{theorem}\label{thm:lclt_det}
  Assume that $(\cG.1)$-$(\cG.4)$ hold.  Let $K \subset \bbR^d$ and $I \subset (0,\infty)$ be compact sets. Then,
  \begin{align}\label{eq:lclt_det}
    \lim_{n \to \infty} \sup_{x \in K} \sup_{t \in I}
    \big| n^d\, q(n^2 t, g_n(x)) - \cC_{\cG}^{-1} k^{\Si}_t(x) \big|
    \;=\;
    0.
  \end{align}  
\end{theorem}
\begin{proof}
  The argument is divided into two steps.  First we derive~\eqref{eq:lclt_det} pointwise in $t$ and $x$.  Subsequently, we extend the convergence to hold uniformly in $t \in I$ and $x \in K$ via a covering argument. 
  
  \textsc{Step~1.} Fix any $x \in \bbR^d$ and $t > 0$.  For $r > 0$ and $n \in \bbN$ let
  \begin{align}\label{eq:J}
    J(n,r)
    \;\ldef\;
    J(n,r,t,x)
    \;=\;
    Q\big[ n^{-1} X_{n^2t} \in C(x,r)\big]
    \,-\, \int_{C(x,r)} k^\Si_t(y)\, \mathrm{d}y.
  \end{align}
  Now we rewrite~\eqref{eq:J} as $J(n,r) = J_1(n,r) + J_2(n,r) + J_3(n,r) + J_4(n,r)$, where
  \begin{align*}
    J_1(n,r)
    &\;=\;
    \sum_{z \in \La_n(x,r)} \big( q(n^2 t, z) - q(n^2 t, g_n(x))\big),
    \\[.5ex]
    J_2(n,r)
    &\;=\;
    \frac{|\La_n(x,r)|}{n^d}
    \Big( n^d \, q\big( n^2 t, g_n(x)\big) \,-\, \cC_{\cG}^{-1} k^\Si_t(x)\Big),
    \\[2ex]
    J_3(n,r)
    &\;=\;
    k^{\Si}_t(x)\,
    \Big( |\La_n(x,r)|\, \cC_{\cG}^{-1}\, n^{-d} \,-\, (2r)^d \Big),
    \\[1.5ex]
    J_4(n,r)
    &\;=\;
    \int_{C(x,r)} \big( k^{\Si}_t(x) \,-\, k^{\Si}_t(y) \big)\, \mathrm{d} y.
  \end{align*}
  By rearranging those terms we get
  \begin{align}\label{eq:target}
    \big| n^d \, q(n^2 t, g_n(x)) \,-\, \cC_{\cG}^{-1}\, k^\Si_t(x) \big|
    \;\leq\;
    \frac{n^d}{|\La_n(x,r)|}\,
    \big(|J| \,+\, |J_1| \,+\, |J_3| \,+\, |J_4|\big).
  \end{align}
  Thus, it suffices to show that the right hand side goes to zero when we first take the limit $n \to \infty$ and then $r \to 0$. First, it follows directly from $(\cG.1)$ and $(\cG.3)$ that
  \begin{align}\label{eq:JandJ3}
    \lim_{n\to \infty}
    \frac{n^d}{|\La_n(x,r)|}\, \big( |J| \,+\, |J_3| \big)
    \;=\;
    0.
  \end{align}
  Moreover, by the continuity of $k^{\Si}_t$, Lebesgue's differentiation theorem and $(\cG.1)$,
  \begin{align}\label{eq:J4}
    \lim_{r \to 0} \lim_{n \to \infty}
    \frac{n^d  J_4}{|\La_n(x,r)|}
    \;=\;
    \lim_{r \to 0}
    \frac{1}{\cC_{\cG} (2r)^d}
    \int_{C(x,r)} \big( k^{\Si}_t(x) \,-\, k^{\Si}_t(y) \big)\, \mathrm{d} y
    \;=\;
    0.
  \end{align}
  We are left with handling the summand involving $|J_1|$.  We begin by comparing $\La_n(x,r)$ with balls in the graph distance.  By $(\cG.1)$ we can find $\overline{n} \in \bbN$ large enough such that $|\La_n(x,r)| > 0$ and $g_n(x) \in \La_n(x,r)$ for all $n \geq \overline{n}$.  It follows from $(\cG.2)$, after possibly choosing a larger $\overline{n}$, that for all $n \geq \overline{n}$ and all $y \in \La_n(x,r)$,
  \begin{align}
    d(y, g_n(x))
    \;\leq\;
    \big( C_4 \, |y - g_n(x)|_{\infty} \big) \vee n^{1-\de}
    \;\leq\;
    (2 r c)\, n.
  \end{align}
  Thus $\La_n(x,r) \subset B\big(g_n(x), (2rc)n\big)$, whenever $n \geq \overline{n}$.  Thus, for all $n \geq \overline{n}$ (which may depend on $r$),
  \begin{align*}
    \frac{n^d |J_1(n,r)|}{|\La_n(x,r)|}
     &\;\leq\;
     \max_{z \in \La_n(x,r)} n^d\, \big| q(n^2 t, z) \,-\, q(n^2 t, g_n(x) \big|
     \\[1ex]  
     &\;\leq\;
     \max_{z \in  B(g_n(x), (2rc)n)} n^d\,
     \big| q(n^2 t, z) \,-\, q(n^2 t, g_n(x) \big|. 
  \end{align*}
  Now an application of $(\cG.4)$ gives
  \begin{align}\label{eq:J1}
    \lim_{r \to 0} \limsup_{n \to \infty} \frac{n^d |J_1(n,r)|}{|\La_n(x,r)|}
    \;\leq\;
    \lim_{r \to 0} c\, \Big(\frac{r}{\sqrt{t}}\Big)^{\vr} t^{-\frac{d}{2}}
    \;=\;
    0.
  \end{align}
  Combining~\eqref{eq:JandJ3},~\eqref{eq:J4} and~\eqref{eq:J1} in~\eqref{eq:target} we get for any fixed $x \in \bbR^d$ and $t  >0$,
  \begin{align}\label{eq:pointwiseCLT}
    \lim_{n \to \infty}
    \big| n^d\, q(n^2 t, g_n(x)) \,-\, \cC_{\cG}^{-1}\, k^{\Si}_t(x) \big|
    \;=\;
    0.
  \end{align}

  \textsc{Step~2.} We now prove the full result using a covering argument. For $\eta \in (0,1) \cap \bbQ$ we define the set $\cX \ldef \{ (y,s)\in (K \times I) \cap (\eta \bbZ^d \times \eta^2 \bbZ) \}$ and for all $x \in K$, $t \in I$ we write $\big( y(x), s(t) \big)$ for a ``closest'' point to $(x,t)$ in $\cX$ so that
  \begin{align}\label{eq:points}
    \big|x - y(x)\big|_{\infty}
    \;\leq\;
    \eta,
    \qquad t \in \big(s(t) - \eta^2, s(t)\big].
  \end{align}
  We know that~\eqref{eq:pointwiseCLT} holds for all $(y,s) \in \cX$.  As $\cX$ is a finite set, for a given $\ve > 0$, we can find $\widetilde{n} \in \bbN $ such that for all $n \geq \widetilde{n}$,
  \begin{align}\label{eq:eps}
    \sup_{(y,s) \in \cX}
    \big| n^d\, q(n^2 s, g_n(y)) \,-\, \cC_{\cG}^{-1} k^{\Si}_s(y) \big|
    \;\leq\;
    \ve,
  \end{align}
  and from $(\cG.4)$ we deduce, after taking $\widetilde{n}$ larger if necessary, that for all $n \geq \widetilde{n}$,
  \begin{align}\label{eq:eps2}
    \sup_{(y,s) \in \cX}
    \sup_{
      \substack{x_1, x_2 \in B(g_n(y), \eta n) \\ s - \eta^2 < s_1,s_2 \leq s}
    } n^d\, \big| q(n^2 s_1, x_1) \,-\, q(n^2 s_2, x_2) \big|
    \;\leq\;
    c\, \Big(\frac{\eta}{\sqrt{T}}\Big)^{\vr}\, T^{-\frac{d}{2}},
  \end{align}
  where $T \ldef \inf I > 0$.  On the other hand, for any $x \in K$, $t \in I$ and $n \geq \widetilde{n}$,
  \begin{align}
    &\big| n^d\, q(n^2 t, g_n(x)) \,-\, \cC_{\cG}^{-1} k^{\Si}_t(x) \big|
    \nonumber\\[.5ex]  
    \label{eq:L1}
    &\mspace{36mu}\leq\;
    \big| n^d\, q(n^2 t, g_n(x)) \,-\, n^d\, q\big(n^2 s(t), g_n(y(x)\big)\big|
    \\[.5ex]
    \label{eq:L2}
    &\mspace{72mu}+\,
    \big|
      n^d\, q\big(n^2 s(t), g_n(y(x))\big)
      \,-\, \cC_{\cG}^{-1} k^{\Si}_{s(t)}(y(x))
    \big|
    \\[.5ex]
    \label{eq:L3}
    &\mspace{72mu}+\,
    \big| \cC_{\cG}^{-1} k^{\Si}_{s(t)}(y(x))- \cC_{\cG}^{-1} k^{\Si}_t(x)\big|.
  \end{align}
  We estimate each term individually.  By means of~\eqref{eq:eps2} we can bound~\eqref{eq:L1} by $\ve$ for $\eta$ small enough.  Clearly,~\eqref{eq:L2} is bounded by $\epsilon$ thanks to~\eqref{eq:eps}.  Finally, the regularity of $k_t^{\Si}(x)$ in space and time, together with~\eqref{eq:points} implies that~\eqref{eq:L3} is bounded by $\ve$ uniformly in $x \in K$ and $t \in I$ for $\eta$ small enough.  Hence, there exists $\widetilde{n} \in \bbN$ such that for all $n \geq \widetilde{n}$,
  \begin{align*}
    \sup_{x \in K} \sup_{t \in I}
    \big| n^d\, q(n^2 t, g_n(x)) \,-\, \cC_{\cG}^{-1} k^{\Si}_t(x) \big|
    \;\leq\;
    3 \ve,
  \end{align*}
  which is the desired conclusion.
\end{proof}
\begin{remark}
  If in addition the on-diagonal estimate 
  \begin{align*}
    n^d\, q(n^2 t,g_n(x)) \;\leq\; c\, t^{-d/2}, \qquad \forall\, t>0,
  \end{align*}
  is available, then~\eqref{eq:lclt_det} can be extended to hold uniformly in $t \in [s,\infty)$ for any fixed $s > 0$.  In fact, in that case both $n^d q(n^2 t,g_n(x))$ and $k^{\Si}_t(x)$ converge to zero as $t \to \infty$. 
\end{remark}

\section{Local CLT for the dynamic RCM on $\bbZ^d$}
\label{sec:dyn_lattice}
In this section we will work again in the setting introduced in Section~\ref{sec:setting_intro}.  We aim at applying Theorem~\ref{thm:lclt_det} to the dynamic RCM to prove Theorem~\ref{thm:dyn_lclt}.  The main step will be the verification of condition $(\cG.4)$ based on the oscillation inequality in Theorem~\ref{thm:osc} and on the fact that, for $\prob$-a.e.\ $\om$, the function $u=p^\om(0,0; \cdot,\cdot)$  satisfies $\partial_tu=\cL^\om_t u$, see Proposition~\ref{prop:forwardeq}. Another ingredient will be the following version of the ergodic theorem.
\begin{prop} \label{prop:krengel_pyke}
  Let
  \begin{align*}
    \cQ
    \;\ldef\;
    \big\{
      I \times B :
      \text{
        $I\subset \bbR$ non-empty compact interval,
        $B$ closed Euclidean ball in $\bbR^d\!$
      }
    \big\}.
  \end{align*}
  Suppose that Assumption~\ref{ass:P} holds.  Then, for any $f \in L^1(\Om)$,
  \begin{align*}
    \lim_{n \to \infty} \sup_{I \times B \in \cQ}
    \bigg|
      \frac{1}{n^{d+2}}\,
      \int_{n^2 I} \sum_{x \in (nB) \cap \bbZ^d}\mspace{-12mu}
        f \circ \tau_{t,x}\;
      \mathrm{d}t
      \,-\,
      |I \times B| \cdot \mean\!\big[f\big]
    \bigg|
    \;=\;
    0, \qquad \prob\text{-a.s.}
  \end{align*}
\end{prop}
\begin{proof}
  For discrete multiparameter processes such a uniform ergodic theorem under standard scaling has been shown, for instance, in \cite[Theorem~1]{KP87} and the corresponding result for continuous parameter processes in \cite[Theorem~2]{KP87}.  The claim, involving different scaling in space and time, follows by the same arguments.
\end{proof}
As a direct consequence from Proposition~\ref{prop:krengel_pyke} we get the following lemma.
\begin{lemma}\label{lem:erg_deterministic} 
  Suppose that Assumptions~\ref{ass:P} and \ref{ass:moment} hold.  Then, $\prob$-a.s., for any $x \in \bbR^d$, $\de \in (0,1)$ and $t \geq \de^2$,
  \begin{align*}
    K_{\mu}
    &\;\ldef\;
    \limsup_{n \to \infty}
    \|1 \vee \mu^{\om}\|_{p, p, n^2[t-\de^2,t] \times B(g_n(x), \de n)}
    \;<\;
    \infty,
    \\[.5ex]
    K_{\nu}
    &\;\ldef\;
    \limsup_{n \to \infty}
    \|1 \vee \nu^{\om}\|_{q, q, n^2[t-\de^2,t] \times B(g_n(x), \de n)}
    \;<\;
    \infty.
  \end{align*}
\end{lemma}
\begin{prop}\label{prop:fromosc}
  Let $\de \in (0,1)$, $\sqrt{t}/2 \geq \de$ and $x \in \bbR^d$ be fixed.  Then, there exist positive constants $C_6$ and $\vr$ only depending on $K_{\mu}$ and $K_{\nu}$ such that
  \begin{align*}
    \limsup_{n \to \infty}
    \sup_{
    \substack{
      x_1, x_2 \in B(g_n(x), \de n ) \\
      t-\de^2 < s_1, s_2 \leq t}
    } n^d\, \big | p^{\om}(0,0;n^2 s_1,x_1) \,-\, p^{\om}(0,0;n^2s_2,x_2)\big|
    \;\leq\;
    C_6 \left(\frac{\de}{\sqrt{t}}\right)^{\!\vr} t^{-\frac{d}{2}}.
  \end{align*}
\end{prop}
\begin{proof}
  This follows from the oscillation inequality in Theorem~\ref{thm:osc} similarly as Corollary~\ref{cor:cont_caloric}.  To apply Theorem~\ref{thm:osc}, choose $t_0 = n^2 t$, $x_0 = g_n(x)$ $p=p'$ and $q=q'$ with $p$ and $q$ from Assumption~\ref{ass:moment}, and  take $\vt \in (0,1/2)$ from Theorem~\ref{thm:osc}.  Set $\de \ldef \vt^k \sqrt{t} /2$, $k \geq 0$ and
  \begin{align*}
    Q_k
    \;\ldef\;
    n^2 \big[t- \de_k^2, t\big] \times B\big(g_n(x),\de_k n\big),
    \qquad k\geq 0.
  \end{align*}
  Choose $k_0 \in \bbN$ such that $\de_{k_0}\geq \de > \de_{k_0+1}$.  Then, $\de_{k} \in [\de, \sqrt{t}]$ for every $k \leq k_0$.  In view of Lemma~\ref{lem:erg_deterministic} we can find $N_7 = N_7(x, t, \de) \in \bbN$ such that for all  $n\geq N_7$,
  \begin{align*}
    \max\big\{ \|1 \vee \mu\|_{p,p,Q_k}, \|1\vee \nu\|_{q,q,Q_k} \big\}
    \;<\;
    2 (K_\mu \vee K_\nu),
    \quad \forall\, k = 0, \ldots, k_0.
  \end{align*}
  It follows that we can apply the oscillation inequality iteratively with a common constant $\bar{\ga} \in (0,1)$ only depending on $K_{\mu}$ and $K_{\nu}$ for all $n \geq N_7$, so that
  \begin{align*}
    \osc_{Q_{k}} p^{\om}(0,0; \cdot,\cdot)
    \;\leq\;
    \bar{\ga}\, \osc_{Q_{k-1}} p^{\om}(0,0; \cdot,\cdot),
    \qquad \forall\, k = 1, \ldots, k_0,
  \end{align*}
  and by iteration
  \begin{align}\label{eq:preHolder2}
    \osc_{Q_{k_0}} p^{\om}(0,0;\cdot,\cdot)
    \;\leq\;
    \bar{\ga}^{k_0}\, \sup_{Q_{0}} p^{\om}(0,0;\cdot,\cdot).
  \end{align}
  Note that $\bar{\ga}^{k_0} \leq c \big(\de / \sqrt{t}\big)^\vr$, for some positive constants $\vr$ and $c$ only depending on $\bar{\ga}$.  Further, we can bound the right hand side of~\eqref{eq:preHolder2} by using the on-diagonal bound in Corollary~\ref{cor:hk_upper},
  \begin{align*}
    n^d\, \sup_{Q_0} \, p^\om(0,0;\cdot,\cdot)
    \;\leq\;
    c\, t^{-d/2}.
  \end{align*}
  Finally, since
  \begin{align*}
    Q_{k_0}
    \;=\;
    n^2 [t -\de_{k_0}^2, t] \times B\big(g_n(x), \de_{k_0} n \big)
    \supset
    n^2 [t-r,t] \times B\big(g_n(x), \de n\big),
  \end{align*}
  the claim follows.
\end{proof}
\begin{proof}[Proof of Theorem~\ref{thm:dyn_lclt}]
  We apply Theorem~\ref{thm:lclt_det}.  Since in the present setting $\cG=\bbZ^d$, conditions ($\cG.1$) and ($\cG.2$) are obviously satisfied. Condition ($\cG.3$) follows from the invariance principle in Theorem~\ref{thm:dyn_ip} established in \cite{ACDS18}.  Finally, Proposition~\ref{prop:fromosc} implies condition ($\cG.4$).
\end{proof}

Finally, we remark that a local limit theorem directly implies a near-diagonal lower heat kernel estimate, which complements the upper bounds obtained in Corollary~\ref{cor:hk_upper} above.

\begin{corro} \label{cor:near diag ests}
Suppose that Assumptions~\ref{ass:P} and \ref{ass:moment} hold. For $\bbP$-a.e.\ $\om$, there exists $N_8(\om)>0$  and $C_7=C_7(d)>0$ such that  for all $t\geq N_8(\om)$ and $x\in B(0,\sqrt{t})$,
\begin{align*}
p^\om(0,t;0,x) \; \geq \; C_7\, t^{-d/2}.
\end{align*}
\end{corro}
\begin{proof}
This follows from Theorem~\ref{thm:dyn_lclt} exactly as  in \cite[Lemma 5.3]{ADS16}.
\end{proof}

\section{Local CLT for the static RCM on random graphs} \label{sec:rg}
As a further application of the oscillation bound in Theorem~\ref{thm:osc} we present in this final section a local limit theorem for the static RCM on a class of random graphs. On $(\bbZ^d, E_d)$ we consider the conductances  $\om=\{\om(e), e\in E_d\} \in \Om \ldef [0,\infty)^{E_d}$, which are now time-independent but possibly taking the value zero. We call an edge $e \in E_d$ \emph{open} if $\om(e) > 0$ and denote by $\cO(\om)$ the set of open edges. We write $x \sim y$ if $\{x,y\} \in \cO(\om)$. Again we equip $\Om$ with a $\si$-algebra $\cF$ and a probability measure $\prob$.
\begin{assumption} \label{ass:P2}
  \begin{enumerate}[(i)]
  \item The law $\prob$ is stationary and ergodic w.r.t.\ space shifts of $\bbZ^d$.

  \item For $\prob$-a.e.\ $\om$, there exists a unique infinite cluster $\cC_{\infty}(\om)$ of open edges.  Moreover, $\prob[0\in \cC_{\infty}] > 0$.  Write $\prob_0[ \,\cdot\, ] \ldef \prob[ \,\cdot \,|\, 0 \in \cC_{\infty}]$ and $\mean_0$ for the expectation w.r.t.\ $\prob_0$.
  \end{enumerate}
\end{assumption} 
For any realization $\om\in \Om$ consider the variable speed random walk (VSRW) $X \equiv (X_t : t \geq 0)$ on $\cC_{\infty}(\om)$ with generator $\cL^{\om}$ acting on bounded functions $f\!: \cC_{\infty}(\om) \to \bbR$ as
\begin{align*}
  \big(\cL^{\om} f)(x)
  \;=\; 
  \sum_{y \sim x} \om(\{x,y\})\, \big(f(y) - f(x)\big).
\end{align*}
Notice that $X$ is reversible with respect to the counting measure. When visiting a vertex $x\in \cC_{\infty}(\om)$, the random walk $X$ waits at $x$ an exponential time with mean $1/\mu^\om(x)$ where $\mu^\om(x)\ldef \sum_{y \sim x} \om(\{x,y\})$, and then it jumps to a vertex $y \sim x$ with probability $\om(\{x,y\})/\mu^\om(x)$. We denote by $\Prob^\om_x$ the quenched law of the process $X$ starting at $x\in \cC_{\infty}(\om)$, and for $x, y \in \cC_{\infty}(\om)$ and $t \geq 0$ let $p^{\om}(t, x, y)$ be the heat kernel of $X$, i.e.\ $p^{\om}(t, x, y) \ldef \Prob_x^{\om}\big[X_t = y\big]$.

In order to state the results, we need to introduce  some further assumptions on the underlying random  graph $(\cC_{\infty}(\om), \cO(\om))$ which require some more notation. We denote by $d^{\om}$ the graph distance on $(\cC_{\infty}(\om), \cO(\om))$, i.e.\ for any $x, y \in \cC_{\infty}(\om)$, $d^{\om}(x,y)$ is the minimal length of a path between $x$ and $y$ that consists only of edges in $\cO(\om)$.  For $x \in \cC_{\infty}(\om)$ and $r \geq 0$, let $B^{\om}(x,r) \ldef {\{ y \in \cC_{\infty}(\om) : d^{\om}(x, y) \leq \lfloor r \rfloor\}}$ be the closed ball with center $x$ and radius $r$ with respect to $d^{\om}$.  Further, for any $ A \subset B \subset \bbZ^d$ we define the \emph{relative} boundary of $A$ with respect to $B$ by
\begin{align*}
  \partial_{\!B}^{\om} A
  \;\ldef\;
  \big\{
    \{x,y\} \in \cO(\om):
    x \in A \,\text{ and }\, y \in B \setminus A
  \big\}.
\end{align*}
\begin{definition}[Regular balls]\label{def:regular}
  Let $C_{\mathrm{V}} \in (0, 1]$, $C_{\mathrm{riso}} \in (0, \infty)$ and $C_{\mathrm{W}} \in [1, \infty)$ be fixed constants.  For $x \in \cC_{\infty}(\om)$ and $n \geq 1$, we say a ball $B^{\om}(x, n)$ is \emph{regular} if it satisfies the following conditions.
  \begin{enumerate}[i)]
  \item Volume regularity of order $d$, i.e.\
    \begin{align*}
      C_{\mathrm{V}}\, n^d \;\leq\; |B^{\om}(x, n)|.
    \end{align*}
  \item (Weak) relative isoperimetric inequality. There exists $\cS^{\om}(x, n) \subset \cC_{\infty}(\om)$ connected such that $B^{\om}(x, n) \subset \cS^{\om}(x, n) \subset B^{\om}(x, C_{\mathrm{W}} n)$ and
    \begin{align*}
      |\partial_{\cS^{\om}(x, n)}^{\om} A|
      \;\geq\;
      C_{\mathrm{riso}}\, n^{-1}\, |A|
    \end{align*}
    for every $A \subset \cS^{\om}(x, n)$ with $|A| \leq \tfrac{1}{2}\, |\cS^{\om}(x, n)|$.
  \end{enumerate}
\end{definition}
\begin{assumption}\label{ass:cluster} 
  For some $\th \in (0,1)$, for $\prob_0$-a.e.\ $\om$, there exists $N_0(\om) < \infty$ such that for all $n \geq N_0(\om)$ the following hold.
  \begin{enumerate}
  \item[(i)] The ball $B^{\om}(0, n)$ is $\th$-\emph{very regular}, that is, the ball $B^{\om}(x, r)$ is regular for every $x \in B^{\om}(0, n)$ and $r \geq n^{\th/d}$.  

  \item[(ii)] There exist $\de > 0$ and $C_8>0$  such that for each $r>0$,
    \begin{align*}
      d^{\om}(y,z)
      \;\leq\;
      \big(C_8 \, |y - z|_\infty\big) \vee n^{1-\de},
      \qquad \forall y,z\in \La_n(x,r).
    \end{align*}
  \end{enumerate}
\end{assumption}
Assumption~\ref{ass:cluster} is satisfied, for instance, on supercritical Bernoulli percolation clusters, see \cite{Ba04}, or clusters in percolation models with long range correlations, see \cite[Proposition~4.3]{Sa17} and \cite[Theorem~2.3]{DRS14}.  Such random graphs have typically a local irregular behaviour, meaning that the required properties in Definition~\ref{def:regular} fail on small scales. In a sense, Assumption~\ref{ass:cluster} provides a uniform lower bound on the radius of regular balls.  For more details and examples we refer to \cite[Examples~1.11--1.13]{DNS18} and references therein.  
\begin{assumption}\label{ass:pq}
  There exist $p, q \in [1, \infty]$ and $\th \in (0, 1)$ satisfying
  \begin{align} \label{eq:cond:pq_rg}
    \frac{1}{p} \,+\, \frac{1}{q}
    \;<\;
    \frac{2(1-\th)}{d-\th},
  \end{align}
  such that for any $e\in E_d$,
  \begin{align*}
    \mean\big[\om(e)^p\big] \;<\; \infty
    \qquad \text{and} \qquad
    \mean\big[\om(e)^{-q} \indicator_{\{e \in \cO\}}\big] \;<\; \infty,
  \end{align*}
  where we used the convention that $0/0 = 0$.
\end{assumption}
\begin{theorem} [QFCLT \cite{DNS18}]\label{thm:ip_rg}
  Suppose there exist $\th \in (0,1)$ and $p,q \in [1,\infty]$ such that Assumptions~\ref{ass:P2}, \ref{ass:cluster}-(i) and \ref{ass:pq} hold.  Then, for $\prob_0$-a.e.\ $\om$, the process $X^{(n)} \equiv \big(X_t^{(n)} \ldef n^{-1} X_{n^2 t}: t\geq 0\big)$, converges (under $\Prob_{\!0}^\om$) in law towards a Brownian motion on $\bbR^d$ with a deterministic non-degenerate covariance matrix $\Si^2$.
\end{theorem}
\begin{theorem}[Quenched local CLT] \label{thm:lclt_rg} 
  Suppose there exist $\th \in (0,1)$ and $p,q \in [1,\infty]$ such that Assumptions~\ref{ass:P2}, \ref{ass:cluster} and \ref{ass:pq} hold.  Then, for any $T_2>T_1>0$ and $K>0$,
  \begin{align*}
    \lim_{n \to \infty} \sup_{|x|\leq K} \sup_{ t\in [T_1, T_2]}
    \big|
      n^d p^{\om}(n^2t, 0,\lfloor nx \rfloor)
      - \prob[0\in \cC_{\infty}]^{-1}k^\Si_t(x)
    \big|
    \;=\;
    0,
    \qquad \prob_0\text{-a.s.,}
  \end{align*}
  with $k_t^\Si$ defined as in \eqref{eq:def_kt}.
\end{theorem}
\begin{remark}
  It appears feasible to derive a local CLT  also for a more general class of speed measures for the random walk rather than only for the VSRW as in Theorem~\ref{thm:lclt_rg}. On $(\bbZ^d,E_d)$ such a result has been shown in \cite{AT19}.
\end{remark}
\begin{proof}[Proof of Theorem~\ref{thm:lclt_rg}]
  The result follows  from Theorem~\ref{thm:lclt_det} once conditions $(\cG.1)-(\cG.4)$ are verified.  For condition $(\cG.1)$ note that for any $r>0$ and $x\in \bbR^d$ by the ergodic theorem in \cite[Theorem~1]{KP87},
  \begin{align*}
    \frac{|\La_n(x,r)|}{(2n r)^d}
    \;=\;
    \frac{1}
    {
      (2n r)^d} \sum_{y \in C(nx,nr)}
      \indicator_{\{ y\in \cC_{\infty}(\om)\}
    }
    \underset{n \to \infty}{\;\longrightarrow\;}
    \prob[0 \in \cC_{\infty}] > 0,
    \qquad \prob\text{-a.s.}
  \end{align*}
  and therefore also $\prob_0$-a.s.  Condition $(\cG.2)$ coincides with Assumption~\ref{ass:cluster}-(ii) and $(\cG.3)$ is a consequence of the invariance principle in Theorem~\ref{thm:ip_rg}.  For condition $(\cG.4)$ we aim to apply Theorem~\ref{thm:osc} together with the ergodic theorem in \cite[Theorem~1]{KP87} (cf.\ Proposition~\ref{prop:krengel_pyke} above for the space-time version), which implies $(\cG.4)$ by the same arguments as in the proof of Proposition~\ref{prop:fromosc} above.  Note that the conductances are constant in time in the present setting, so we may choose $p'=q'=\infty$ in Theorem~\ref{thm:osc} and \eqref{eq:cond:pq} reduces to \eqref{eq:cond:pq_rg}.

It remains to check that the graph $(\cC_{\infty}(\om), \cO(\om))$ satisfies  Assumption~\ref{ass:graph}.  Obviously, for every $x\in \cC_{\infty}(\om)$ and $n \geq 1$, the ball $B^{\om}(x,n)$ is contained in the corresponding ball with respect to the graph distance.  Thus, $|B^\om(x,n)| \leq c n^d$, and the volume regularity in Assumption~\ref{ass:graph}-(i) follows from this and Assumption~\ref{ass:cluster}-(i).  Furthermore,  Assumption~\ref{ass:cluster}-(i) also implies an isoperimetric inequality on large sets (see \cite[Lemma~2.10]{DNS18}), which in conjunction with the volume regularity implies the Sobolev inequality in Assumption~\ref{ass:graph}-(ii) with  $d' \ldef (d-\th)/(1-\th)$, see \cite[Proposition~3.5]{DNS18}. Finally, the weak Poincar\'{e} inequality in  Assumption~\ref{ass:graph}-(iii) follows from the  relative isoperimetric inequality provided by Assumption~\ref{ass:cluster}-(i) by applying a discrete co-area formula, see \cite[Lemma~3.3.3]{Sa96} and cf.\ Remark~\ref{rem:ass_graph} above.
\end{proof}

Similarly as in Corollary~\ref{cor:near diag ests} above, one can derive near-diagonal heat kernel estimates from Theorem~\ref{thm:lclt_rg} following the argument in \cite[Lemma~5.3]{ADS16}.

\appendix
\section{A technical estimate} \label{sec:tech}
\begin{lemma} \label{lem:A1}
  Let $g \in C^1((0, \infty))$ be a convex, non-increasing function. Assume that $g'$ is piecewise differentiable and that there exists $\ga \in (0, 1]$ such that $\ga g'(r)^2 \leq g''(r)$ for a.e.\ $r \in (0, \infty)$.  Then, for all $x, y > 0$ and $b, a \geq 0$,
  \begin{align}\label{eq:A1}
    &-\big(b^2 g'(y) - a^2 g'(x) \big)(y-x)
    \nonumber\\[.5ex]
    &\mspace{36mu}\leq\;
    \begin{cases}
      -
      \dfrac{\ga}{2}\, \big(a^2 \wedge b^2\big)\,
      \big( g(y) - g(x) \big)^{2}
      +
      \dfrac{2}{\ga}\,
      \bigg(\dfrac{a^2}{b^2} \vee \dfrac{b^2}{a^2}\bigg)
      \big(b - a\big)^2,
      & a \wedge b > 0,\\[2ex]
      \big( \!-\mspace{-4mu}x g'(x) \vee -y g'(y)\big)\, (b-a)^2,
      & a \wedge b = 0.
    \end{cases}
  \end{align}
\end{lemma}
\begin{proof}
  Since $g$ is non-increasing, the case $a=0$ or $b=0$ is immediate.  In the sequel, we assume that $a \wedge b > 0$.  First, notice that an application of the Cauchy-Schwarz inequality yields for any $x,y > 0$,
  \begin{align*}
    \ga \bigg( \int_x^y\! g'(t)\, \md t \bigg)^{\!\!2}
    \;\leq\;
    \ga \bigg( \int_x^y\! g'(t)^2\, \md t \bigg)
    \bigg( \int_x^y\! 1\, \md t \bigg)
    \;\leq\;
    \bigg( \int_x^y\! g''(t)\, \md t \bigg)
    \bigg( \int_x^y\! 1\, \md t \bigg),
  \end{align*}
  where we used in the second step that $\ga g'(t)^2 \leq g''(t)$ for a.e.\ $t \in (0, \infty)$.  Hence,
  \begin{align}
    \ga\, \big(g(y) - g(x)\big)^2
    \;\leq\;
    \big(g'(y) - g'(x)\big)(y-x).
  \end{align}
  Without loss of generality, assume that $y > x$. Since $g$ is convex and non-increasing, it follows that $0 \leq -g'(y) (y-x) \leq g(x) - g(y)$.  Hence,
  \begin{align*}
    &-\big(b^2 g'(y) - a^2 g'(x) \big)(y-x)
    \\[.5ex]
    &\mspace{36mu}=\;
    -a^2 \big(g'(y) - g'(x) \big)(y-x) \,-\, g'(y)(y-x)(a+b)(b-a)
    \\[.5ex]
    &\mspace{36mu}\leq\;
    -\ga \big(a^2 \wedge b^2\big)\, \big( g(y) - g(x) \big)^2
    \,+\, 2 \big( g(x) - g(y) \big) (a \vee b) |b-a|.
  \end{align*}
  Thus, by applying the Young inequality, that reads $|\al \be| \leq \frac{1}{2}(\ve\, \al^2 + \be^2/\ve)$ for any $\ve > 0$, we obtain
  \begin{align*}
    &-\big(b^2 g'(y) - a^2 g'(x) \big)(y-x)
    \\[.5ex]
    &\mspace{36mu}\leq\;
    - \Big( \ga \big(a^2 \wedge b^2\big) - \ve (a^2 \vee b^2) \Big)\,
    \big( g(y) - g(x) \big)^2
    \,+\,
    \frac{1}{\ve}\, (b-a)^2.
  \end{align*}
  By choosing $\ve = \frac{\ga}{2}(a^2 \wedge b^2) / (a^2 \vee b^2)$, the estimate \eqref{eq:A1} follows.
\end{proof}

\section{Forward and backward equations for the  semigroup} \label{sec:be_fe}
In this section we aim to verify that also in the case of dynamic unbounded conductances the heat kernel of the random walk satisfies the forward equation. We will work in the setting outlined in Section~\ref{sec:setting_intro}. In particular we will suppose throughout that Assumption~\ref{ass:P} holds, which  implies that, $\prob$-a.s., the conductances are local integrable in time, that is
\begin{align}
  \label{eq:local:integrability}
  \prob\!\bigg[ \int_I \om_s(x,y)\, \md s < \infty \bigg]
  \;=\;
  1,
  \qquad\text{for every finite interval } I \subset \bbR.
\end{align}
For any $s \geq 0$, we write $(P_{s,t}^{\om} : t \geq s)$ for the Markov semigroup associated with the random walk $X$, i.e.\ $(P_{s,t}^{\om}f)(x) = \Mean_{s,x}^{\om}[f(X_t)]$ for any bounded function $f\!: \bbZ^d \to \bbR$, $0 \leq s < t$ and $x \in \bbZ^d$. Further, we write $(P_{s,t}^\om)^*$ for the adjoint of $P^\om_{s,t}$ in $\ell^2(\bbZ^d)$. The associated heat kernel is still denoted $p^{\om}(s,x;t,y) \ldef \Prob_{s,x}^{\om}\!\big[X_t = y\big]$ for $x, y \in \bbZ^d$ and $0 \leq s < t$.  As a consequence of \eqref{eq:shift:space-time} we have
\begin{align}\label{eq:shift:hk}
  p^{\tau_{h, z} \om}(s,x;t,y)
  \;=\;
  p^{\om}(s+h,x+z; t+h, y+z).
\end{align}

Next we briefly recall the construction of the time-inhomogeneous Markov process $X$ starting at time $s \geq 0$ in $x \in \bbZ^d$, cf.~\cite[Section~4]{ACDS18}.  Let $(E_n : n \in \bbN)$ be a sequence of independent $\mathop{\mathrm{Exp}}(1)$-distributed random variables.  Further, set $\pi_t^{\om}(x,y) \ldef \om_t(x,y) / \mu_t^{\om}(x) \indicator_{\{(x,y) \in E_d\}}$, where $\mu_t^{\om}(x) \ldef \sum_{y : (x,y) \in E_d} \om_t(x,y)$ for any $t \in \bbR$, $x \in \bbZ^d$.  We specify both the sequence of jump times, $(J_n : n \in \bbN_0)$ and positions, $(Y_n : n \in \bbN_0)$, inductively.  For this purpose, set $J_0 = s$ and $Y_0 = x$.  Suppose that, for any $n \geq 1$, we have already constructed the random variables $(J_0, Y_0, \ldots, J_{n-1}, Y_{n-1})$.  Then, $J_{n}$ is given by
\begin{align*}
  J_{n}
  \;=\;
  J_{n-1} \,+\,
  \inf\bigg\{
    t \geq 0
    \;:\;
    \int_{J_{n-1}}^{J_{n-1} + t} \mu_u^{\om}(Y_{n-1})\, \md u \geq E_n
  \bigg\},
\end{align*}
and at the jump time $J_n$ the distribution of $Y_n$ is given by $\pi_{J_n}^{\om}(Y_{n-1}, \cdot)$.  Since, under Assumption~\ref{ass:P}, $\sup_{n \in \bbN_0} J_n = \infty$, $\prob$-a.s., the Markov process $X$ is given by
\begin{align*}
  X_t \;=\; Y_n \quad \text{on} \quad [J_n, J_{n+1})
  \qquad \forall\, n \in \bbN_0.
\end{align*}
Note that, under $\Prob_{s,x}^{\om}$, we have $J_0 = s$ and $Y_0 = x$ almost surely,  the conditional law of $J_n$ given $(J_0, Y_0, \ldots, J_{n-1}, Y_{n-1})$ (also called survival distribution with time-dependent hazard rate $\mu^{\om}_t(Y_{n-1})$) is
\begin{align*}
  \mu_t^{\om}(Y_{n-1})\, \me^{-\int_{J_{n-1}}^t \mu_u^{\om}(Y_{n-1})\, \md u}\,
  \indicator_{\{t > J_{n-1}\}}\, \md t,
\end{align*}
and the conditional law of $Y_n$ given $(J_0, Y_0, \ldots, J_{n-1}, Y_{n-1}, J_n)$ is $\pi_{J_n}^{\om}(Y_{n-1}, \cdot)$.

Further, for the Markov process $X$ as constructed above  the  strong Markov property holds, an application of which yields that $(P_{s,t}^{\om} : t \geq s)$ satisfies the integrated backward equation, that is, for $\prob$-a.e.\ $\om$,
\begin{align}
  \label{eq:backward:integrated}
  (P_{s,t}^{\om} f)(x)
  \;=\;
  \me^{-\int_s^t \mu_u^{\om}(x)\, \md u} f(x)
  \,+\,
  \int_s^t
    \me^{-\int_s^r \mu_u^{\om}(x)\, \md u}\!
    \sum_{y:(x,y) \in E_d}\mspace{-9mu} \om_r(x,y)\, (P_{r,t}^{\om} f)(y)\,
  \md r
\end{align}
for any $f \in \ell^{\infty}(\bbZ^d)$, $0 \leq s < t < \infty$ and $x \in \bbZ^d$.
\begin{prop} \label{prop:be}
  For $\prob$-a.e.\ $\om$, every  $x, y \in \bbZ^d$  and $f \in \ell^{\infty}(\bbZ^d)$ the following hold.
  \begin{itemize}
  \item[(i)] For every $t > 0$, the map $ s \mapsto p^{\om}(s,x;t,y)$ is differentiable at almost every $s \in (0,t)$.  In particular, $\lim_{s \uparrow t} p^{\om}(s,x;t,y) = p^{\om}(t,x;t,y) = \indicator_y(x)$.

  \item[(ii)]  For every $t > 0$,
    \begin{align*}
      - \partial_s (P_{s,t}^{\om} f)(x)
      \;=\;
      \big(\cL_s^{\om} (P_{s,t}^{\om} f)\big)(x),
      \qquad \text{for a.e.\ $s \in (0,t)$}.
    \end{align*}
    % 
  % \item[(iii)]  For any $s \geq 0$ and $\prob$-a.e.\ $\om$, the map $(s,\infty) \ni t \mapsto p_{s,t}(x,y)$ continuous. In particular, $\lim_{t \downarrow s} p_{s,t}^{\om}(x,y) = p_{s,s}^{\om}(x,y) = \indicator_y(x)$.
  %   % 
  % \item[(iv)] For a.e.\ $s \in [0, \infty)$.
  %   %
  %   \begin{align*}
  %     \lim_{h \downarrow 0}
  %     \frac{1}{h}\, \big( (P_{s,s+h}^{\om}f)(x) - f(x) \big)
  %     \;=\;
  %     (\cL_s^{\om}f)(x).
  %   \end{align*}
  \end{itemize}
\end{prop}
\begin{proof}
  (i) We will show that for every $t > 0$, $x, y \in \bbZ^d$ and $\prob$-a.e.\ $\om$ the mapping $[0,t) \ni s \mapsto p^{\om}(s,x;t,y)$ is absolute continuous, which implies (i).  For that purpose, fix some $f \in \ell^{\infty}(\bbZ^d)$ and $t > 0$.  Since, for every $x \in \bbZ^d$, the map $t \mapsto \mu_t^{\om}(x)$ is $\prob$-a.s.\ locally integrable by \eqref{eq:local:integrability}, the absolute continuity of the Lebesgue integral implies that, for every $\ve > 0$, there exists $\de \equiv \de(x) > 0$ such that
  \begin{align*}
    \int_D \mu^{\om}_u(x)\, \md u
    \;\leq\;
    \frac{\ve}{2 \|f\|_{\infty}}
    \qquad \forall\, D \in \cB(\bbR)
    \text{ with Lebesgue measure less than } \de.
  \end{align*}
  Then, by using the integrated backward equation \eqref{eq:backward:integrated} and the Cauchy-Schwarz inequality,
  \begin{align*}
    &\sum_{i=1}^n\,
    \big|(P_{s_i,t}^{\om} f)(x) - (P_{r_i,t}^{\om}f)(x)\big|
    \\
    &\mspace{36mu}\leq\; 
    2 \|f\|_{\infty}\,
    \sum_{i=1}^n
    \Big(1 - \me^{-\int_{r_i}^{s_i} \mu_u^{\om}(x)\, \md u} \Big)
    %\\
    \;\leq\;
    2 \|f\|_{\infty}\,
    \bigg(\int_D \mu_u^{\om}(x)\, \md u \bigg)
    \;<\;
    \ve
  \end{align*}
  for any union $D = \bigcup_{i=1}^n (r_i, s_i)$ of pairwise disjoint intervals $(r_i, s_i) \subset [0,t]$ of total length less than $\de$.

  (ii) We rewrite the right-hand side of \eqref{eq:backward:integrated} as
  \begin{align*}
    \me^{-\int_s^t \mu_u^{\om}(x)\, \md u}\, f(x)
    \,+\,
    \me^{\int_0^s \mu_u^{\om}(x)\, \md u}\,
    \int_s^t
    \me^{-\int_0^r \mu_u^{\om}(x)\, \md u}
    \sum_{y:(x,y) \in E^d}\mspace{-9mu} \om_r(x,y) (P_{r,t}^{\om} f)(y)\,
    \md r.
  \end{align*}
  Since $t \mapsto \mu_t^{\om}(x)$ is locally integrable, the differential form of the backward equation in weak sense follows from \cite[Theorem~6.3.6]{Co13} together with an application of the chain and product rule.
  %
% (iii)   Again, this can be deduced from the absolute continuity of the Lebesgue integral.  Indeed, $\prob$-a.s., for any $f \in \ell^{\infty}(\bbZ^d)$ and $r,t \in (s, \infty)$ with $r < t$,
%  %
%  \begin{align*}
%    & \big|(P_{s,r}^{\om} f)(x) - (P_{s,t}^{\om}f)(x)\big|
%    \;=\;
%    \big|
%      \big(P_{s,r}^{\om} (f - P_{r,t}^{\om} f) \big)(x)
%    \big| \\
%   & \mspace{36mu} \;\leq\;
%    2 \|f\|_{\infty}\,
%    \sum_{y \in \bbZ^d} p_{s,r}^{\om}(x,y)
%    \Big(1 - \me^{-\int_r^t \mu_u^{\om}(y)\, \md u}\Big).
%  \end{align*}
%  %
%  Thus, by applying Lebesgue's dominated convergence theorem, it follows that
%  %
%  \begin{align*}
%    \lim_{t \downarrow r} |(P_{s,r}^{\om} f)(x) - (P_{s,t}^{\om} f)(x)| \;=\; 0.
%  \end{align*} 
%  
%  (iv)  In view of \eqref{eq:shift:hk}, it suffices to consider the case $s=0$. Note that, for any $x \in \bbZ^d$ and $f \in \ell^{\infty}(\bbZ^d)$,
%  %
%  \begin{align*}
%    \frac{1}{h}\big( (P_{0,h}^{\om}f)(x) - f(x) \big)
%    &\;=\;
%    \frac{1}{h}  
%    \int_0^h
%      \me^{-\int_0^r \mu_u^{\om}(x)\, \md u}\, (\cL_r^{\om} f)(x)
%    \md r
%    \\  
%    & \quad+
%    \frac{1}{h}
%    \int_0^h
%      \me^{-\int_0^r \mu_u^{\om}(x)\, \md u}\,
%      \sum_{y} \om_r(x,y)
%      \big((P_{r,h}^{\om}f)(y) - P_{h,h}^{\om}f(y)\big)
%    \md r.
%  \end{align*}
%  %
%  Suppose that 0 is in the intersection of the corresponding Lebesgue sets.  Then, there assertion follows from the absolute continuity of $s \mapsto (P_{s,t}^{\om} f)(y)$ and \cite[Theorem~6.3.6]{Co13}.
\end{proof}
\begin{lemma} \label{lem:time_inverse}
  Define $\tilde{\om}_t(e) \ldef \om_{-t}(e)$ for any $t \in \bbR$ and $e \in E_d$.  Then,
  \begin{align}\label{eq:hk:time_inverse} 
    p^{\om}(0,x;t,y)
    \;=\;
    p^{\tilde{\om}}(-t, y; 0,x),
    \qquad \forall\, x, y \in \bbZ^d, \quad t \geq 0.
  \end{align}
\end{lemma}
\begin{proof}
  Write $B_n \ldef B(0,n)$ and $\tau_{B_n} \ldef \inf\{t \geq 0: X_t \in B_n^c \}$ with $\inf \emptyset \ldef \infty$.  We denote by $p^{\om,n}(s,x;t,y) \ldef \Prob_{s,x}^{\om}\!\big[X_t = y, \, t<\tau_{B_n} \big]$ the heat kernel associated with the process $X$ killed upon exiting $B_n$, and we write $(P_{s,t}^{\om,n}: t \geq s)$ for the transition semigroup.  Recall that the associated time-dependent generator, still denoted by $\cL^{\om}_t$, is acting on functions with Dirichlet boundary condition.  By similar arguments as in Proposition~\ref{prop:be} one can establish a backward equation for $(P_{s,t}^{\om,n} : t \geq s)$, which gives
  \begin{align*}
    &\partial_s \scpr{P_{-s,0}^{\tilde{\om},n} g}{P_{s,t}^{\om, n} f}{B_n}
    \\[1ex]
    &\mspace{36mu}=\;
    \scpr{\cL_{-s}^{\tilde{\om}} P_{-s,0}^{\tilde{\om},n} g}
         {P_{s,t}^{\om,n} f}{B_n} 
    -
    \scpr{ P_{-s,0}^{\tilde{\om},n} g }{\cL_s^{\om} P_{s,t}^{\om,n} f}{B_n}
    \;=\;
    0,
  \end{align*}
  where we used in the last step that $\cL_{-s}^{\tilde{\om}} = \cL^{\om}_s$.  By integration over $[0,t]$ we get
  \begin{align*}
    \scpr{P_{-t,0}^{\tilde{\om},n} g}{f}{B_n}
    - \scpr{g}{P_{0,t}^{\om,n} f}{B_n}
    \;=\;
    0,
  \end{align*}
  and by choosing $f = \indicator_{\{y\}}$ and $g = \indicator_{\{x\}}$ we obtain $p^{\tilde{\om},n}(-t,y;0,x) = p^{\om,n}(0,x;t,y)$.  Finally, since $\lim_{n \to \infty} p^{\om,n}(s,x;t,y) = p^{\om}(s,x;t,y)$ for all $x, y \in \bbZ^d$, $t \geq s$ and $\om \in \Omega$, the result follows by taking the limit $n \to \infty$.
\end{proof}
\begin{prop} \label{prop:forwardeq}
  For $\prob$-a.e.\ $\om$, every $x \in \bbZ^d$ and finitely supported $f\!:\bbZ^d \to \bbR$, the map $ t \mapsto (P_{0,t}^{\om}f)(x)$ is differentiable at almost every $t \in (0,\infty)$ and
  \begin{align}\label{eq:diff:forward}
    \partial_t (P_{0,t}^{\om} f)(x)
    \;=\;
    \big(\cL_t^{\om} (P_{0,t}^{\om} f)\big)(x),
    \qquad \text{for a.e. } t \in (0,\infty).
  \end{align}
  In particular, for $\prob$-a.e.\ $\om$, the function $(t,x) \mapsto u(t,x) = p^{\om}(0,0;t,x)$ solves
  \begin{align*}
    \partial_t u(t,x) \;=\; (\cL_t^{\om} u(t,\cdot))(x),
    \qquad \forall\, x \in \bbZ^d \text{ and a.e. } t \in (0,\infty).
  \end{align*}
\end{prop}
\begin{proof}
  This follows from the backward equation  in Proposition~\ref{prop:be} and Lemma~\ref{lem:time_inverse}. Indeed, let $\tilde \om$ be defined as in Lemma~\ref{lem:time_inverse}, then we have for any $f,g \in \ell^2(\bbZ^d)$, 
  \begin{align*}
    \scpr{P_{0,t}^{\om} f}{g}{\bbZ^d}
    \;=\;
    \scpr{f}{(P_{0,t}^{\om})^*g}{\bbZ^d}
    \overset{\!\!\eqref{eq:hk:time_inverse}\!\!}{\;=\;}
    \scpr{f}{P_{-t,0}^{\tilde{\om}}g}{\bbZ^d}.
  \end{align*}
  Thus, for any finitely supported $f\!: \bbZ^d \to \bbR$ and $g = \indicator_x$, we obtain
  \begin{align*}
    \partial_t (P_{0,t}^{\om} f)(x)
    &\;=\;
    \lim_{h \to 0} \frac{1}{h}
    \scpr{P_{0,t+h}^{\om} f - P_{0,t}^{\om} f}{g}{\bbZ^d}
    \\
    &\;=\;
    \lim_{h \to 0} \frac{1}{h}
    \scpr{f}{P_{-(t+h),0}^{\tilde{\om}} g - P_{-t,0}^{\tilde{\om}} g}{\bbZ^d}
    \;=\;
    \scpr{f}{\partial_t P_{-t,0}^{\tilde{\om}} g}{\bbZ^d}.
  \end{align*}
  Hence, by using the differential backward equation, we get
  \begin{align*}
    \partial_t (P_{0,t}^{\om} f)(x)
    \;=\;
    \scpr{f}{\cL_{-t}^{\tilde{\om}}(P_{-t,0}^{\tilde{\om}} g)}{\bbZ^d}
    \;=\;
    \scpr{f}{\cL_t^{\om} ((P_{0,t}^{\om})^*g)}{\bbZ^d}
    \;=\;
    P_{0,t}^{\om}(\cL_t^{\om} f)(x),
  \end{align*}
  which yields \eqref{eq:diff:forward}.  Finally, consider the function $u(t,x) \ldef p^{\om}(0,0;t,x)$.  Then, by applying \eqref{eq:diff:forward}, we find that
  \begin{align*}
    \partial_t u(t,x)
    \;=\;
    \partial_t (P_{0,t}^{\om} \indicator_x)(0)
    &\;=\;
    P_{0,t}^{\om}(\cL_t^{\om} \indicator_x)(0)
    \\
    &\;=\;
    \scpr{u(t,\cdot)}{\cL_t^{\om} \indicator_x}{\bbZ^d}
    % \;=\;
    % \scpr{\cL_t^{\om} u(t, \cdot)}{\indicator_x}{\bbZ^d}
    \;=\;
    \big(\cL_t^{\om} u(t,\cdot)\big)(x),
  \end{align*}
  which concludes the proof.
\end{proof}
\bibliographystyle{abbrv}
\bibliography{literature}

\end{document}